\documentclass[11pt,twoside]{amsart}
\usepackage{amssymb}
\usepackage{graphicx,color}
\usepackage[all]{xy}

\newtheorem{theorem}{Theorem}[section]
\newtheorem{proposition}[theorem]{Proposition}

\newtheorem{lemma}[theorem]{Lemma}
\newtheorem{corollary}[theorem]{Corollary}
\newtheorem{definition}[theorem]{Definition}
\newtheorem{notation}[theorem]{Notation}
\newtheorem{example}[theorem]{Example}
\newtheorem{remark}[theorem]{Remark}

\newcommand{\skipit}[1]{{}}
\newcommand{\prfend}{\hbox to7pt{\hfil}
\par\vskip-\baselineskip\hbox to\hsize
{\hfil\vbox {\hrule width6pt height6pt}}\vskip\baselineskip}

\newcommand {\PP}{\mathbb{P}}

\newcommand{\cL}{\mathcal{L}}
\newcommand{\cA}{\mathcal{A}}
\newcommand{\cB}{\mathcal{B}}
\newcommand{\cH}{\mathcal{H}}
\newcommand{\cO}{\mathcal{O}}
\newcommand{\cE}{\mathcal{E}}
\newcommand{\cT}{\mathcal{T}}

\DeclareMathOperator{\supp}{supp}

\DeclareMathOperator{\indeg}{indeg}

\DeclareMathOperator{\Der}{Der}

\DeclareMathOperator{\Proj}{Proj}

\DeclareMathOperator{\syz}{syz}

\DeclareMathOperator{\exprk}{exprk}

\newcommand{\myarrow}[2]{\hbox to #1pt{\hfil$\to$\hfil}{\hskip-#1pt{\raise
10pt\hbox to#1pt{\hfil$\scriptscriptstyle #2$\hfil}}}}

\begin{document}

\title{Hypertetrahedral arrangements}

\author[Liena Colarte]{Liena Colarte}
\address{Department de matem\`{a}tiques i Inform\`{a}tica, Universitat de Barcelona, Gran Via de les Corts Catalanes 585, 08007 Barcelona,
Spain}
\email{liena.colarte@ub.edu}

\author[Laura Costa]{Laura Costa}
\address{Department de matem\`{a}tiques i Inform\`{a}tica, Universitat de Barcelona, Gran Via de les Corts Catalanes 585, 08007 Barcelona,
Spain}
\email{costa@ub.edu}

\author[Simone Marchesi]{Simone Marchesi}
\address{Department de matem\`{a}tiques i Inform\`{a}tica, Universitat de Barcelona, Gran Via de les Corts Catalanes 585, 08007 Barcelona,
Spain}
\email{marchesi@ub.edu}

\author[Rosa M. Mir\'o-Roig]{Rosa M. Mir\'o-Roig}
\address{Department de matem\`{a}tiques i Inform\`{a}tica, Universitat de Barcelona, Gran Via de les Corts Catalanes 585, 08007 Barcelona,
Spain}
\email{miro@ub.edu}

\author[Marti Salat]{Marti Salat}
\address{Department de matem\`{a}tiques i Inform\`{a}tica, Universitat de Barcelona, Gran Via de les Corts Catalanes 585, 08007 Barcelona,
Spain}
\email{marti.salat@ub.edu}

\begin{abstract}
In this paper, we introduce the notion of a complete hypertetrahedral arrangement $\cA$ in $\PP^{n}$. We address two basic problems. First, we describe the local freeness of $\cA$ in terms of smaller complete hypertetrahedral arrangements and graph theory properties, specializing the Musta\c{t}\u{a}-Schenck criterion. As an application, we obtain that general complete hypertetrahedral arrangements are not locally free. In the second part of this paper, we bound the initial degree of the first syzygy module of the Jacobian ideal of $\cA$.
\end{abstract}

\thanks{Acknowledgements:   The first, second and fourth authors  are partially   supported
by  MTM2016--78623-P. The fifth author is supported by MDM-2014-0445-18-2}


\maketitle

\tableofcontents

\markboth{}{}

\today

\large

\section{Introduction}

The study of the module $\Der(-\log\,\cA )$ of logarithmic vector fields tangent to the reduced divisor $D_{\cA}$ of a hyperplane arrangement $\cA = \{H_{1},\hdots, H_{m}\}$ began with Saito in \cite{S} and Terao in \cite{Terao}. Since then this topic has been further developed, as it can be seen in \cite{ADS}, \cite{Dimca16}, \cite{Dolgachev-Kapranov}, \cite{Edelman-Reiner}, \cite{MV},  \cite{Mustata-Schenck} or \cite{Ziegler}. A great portion of these contributions seek to determine the algebraic structure of $\Der(-\log\,\cA )$. This is equivalent to describe the first syzygy module $\syz(J_{\cA})$ of the Jacobian ideal $J_{\cA}$ associated to a defining equation of $D_{\cA}$. In this direction, many efforts have been focused on the freeness and local freeness of $\Der(-\log\, \cA )$, and on the initial degree of $\syz(J_{\cA})$.

In this paper, we tackle these questions for a new family of hyperplane arrangements that we call hypertetrahedral. Denote by $e_0=(1:0:\hdots :0)$, $\hdots, $ $e_n=(0:0:\hdots :0:1)$ the vertices  of a $n$-dimensional simplex. A {\em hypertetrahedral arrangement} $\cA \subset \PP^{n}$ is a hyperplane arrangement consisting of hyperplanes passing through the $\binom{n+1}{2}$ linear subspaces
$\mathcal{L}_{i,j}\subset \PP^n$ of codimension 2 defined by the vertices $e_0, \hdots, \widehat{e_{i}}, \hdots, \widehat{e_{j}}, \hdots, e_{n}$ with $i < j$. We say that $\cA$ is {\em complete} if all coordinate hyperplanes belong to $\cA$. We say that a hypertetrahedral arrangement is {\em general} if any intersection outside the $n$-simplex has minimal dimension.
 For $n = 2$, hypertetrahedral arrangements coincide with the family of triangular arrangements introduced in \cite{MV}. Graphic arrangements and the Fermat arrangement are other examples of hypertetrahedral arrangements.

In \cite{S} it was proved that for any hyperplane arrangement $\cA$, its module of derivations $\Der(-\log\,\cA)$ is reflexive. This grants the freeness of hyperplane arrangements in $\PP^{1}$, and the local freeness of line arrangements in $\PP^{2}$. In general, locally free hyperplane arrangements were studied in \cite{Mustata-Schenck}, where the authors provided a nice characterization of them. We specialize this result to our family of hypertetrahedral arrangements. We obtain a local freeness criterion involving only smaller dimensional hypertetrahedral arrangements and graph theory properties. As a remarkable consequence of this result, we prove that complete general hypertetrahedral arrangements are not locally free. Our second goal is to study the generators of the first syzygy module $\syz(J_{\cA})$ of the Jacobian ideal associated to a hypertetrahedral arrangement. We provide upper and lower bounds for the initial degree of $\syz(J_{\cA})$, which are sharp for triangular arrangements. Moreover, the lower bound we give turns out to be sharp for large families of hypertetrahedral arrangements in any dimension.

\vspace{0.3cm}
Let us outline how this work is organized. Section \ref{defs and prelim results} contains the basic definitions and results about hyperplane arrangements needed in the rest of this paper. The main body of this article is divided in the remaining two sections. In Section \ref{local freeness of Hypertetrahedral arrangements}, we define hypertetrahedral arrangements and we study the local freeness of its module of derivations. Let $\cA \subset \PP^{n}$ be a complete hypertetrahedral arrangement with intersection lattice $L(\cA)$. We associate to any $X \in L(\cA)$ a graphic arrangement $\cA_{\Gamma_{X}}$ and a smaller dimensional complete hypertetrahedral arrangement $\cA_{W_{X}}$. The main result Theorem \ref{Theorem: Main Locally Free} proves that $\cA$ is  locally free at $X$ if and only if $\cA_{W_{X}}$ is free and the graph $\Gamma_{X}$ is chordal. By means of this criterion, we prove that general complete hypertetrahedral arrangements are not locally free. 

The last section is devoted to the generators of the first syzygy module $\syz(J_{\cA})$. We give lower and upper bounds for the initial degree of $\syz(J_{\cA})$ and we present families of hypertetrahedral arrangements reaching these bounds.
The lower bound is found in Theorem \ref{ThmBound}, altogether with a set of linear equations describing the generators of $\syz(J_{\cA})$. In the last part of this section, we further develop these equations for triangular arrangements. In Theorem \ref{trunyo}, we apply them to determine the initial degree for $\syz(J_{\cA})$ for any triangular arrangement.

\vskip 5mm \noindent
 {\bf Acknowledgements.}
 The first and last authors are grateful to Prof. Roberta di Gennaro for useful discussions on hyperplane arrangements. All the authors were partially supported by PID2020-113674GB-I00.

\section{Preliminaries}
\label{defs and prelim results}
We fix $k$ an algebraically closed field of characteristic zero, $R=k[x_0,\hdots ,x_n]$ and $\PP^n=\Proj(R)$. We set $e_i=(0:\hdots :1:\hdots :0)$ and $L_i$ the hyperplane defined by $x_i=0$, for $0\le i \le n$. For any homogeneous polynomial $f\in R_d$ of degree $d$, we denote by $J_f$ the Jacobian ideal generated by the partial derivatives $\partial_{x_{j}}f$ of $f$ with respect to $x_j$, $j = 0, \hdots , n$. For any graded $R$-module $M$, we denote by $\indeg(M)$ the {\em initial degree} of $M$, that is, the minimum degree of a nonzero element in $M$. By
$\Der_k(R)$ we denote the free $R$-module of rank $n+1$ generated by the partial derivatives $\partial_{x_{i}}$, $i = 0,\hdots,n$.

Next we recall some basic notions about hyperplane arrangements, for further details see for example \cite{Orlik-Terao}.
A {\em hyperplane arrangement} $\cA = \{ H_1, \hdots ,H_m \} $ in
$\PP ^n$ is
a collection of $m$ distinct hyperplanes of $\PP^n$. Any subcollection $\cB \subset \cA$ is called a {\em subarrangement}. In particular we denote by $\cE_{n}=\emptyset$ the {\em empty arrangement} in $\PP^{n}$.
The hyperplane arrangement divisor $D_{\cA}$ is
defined as $D_{\cA} =\bigcup _{i=1}^m H_i$. If we denote by $f_i$ a linear form defining the hyperplane $H_i$, then $f_{\cA} =\prod _{i=1}^m
 f_i$ is an equation of $D_{\cA}$, which we call a {\em defining equation} for $\cA$. We set $f_{\cE_{n}}:=1$ for the empty arrangement.

Given a hyperplane arrangement $\cA$ in $\PP^{n}$, we define the {\em intersection lattice} $L(\cA)$ of $\cA$ as follows:
$$
L(\cA)=\{H_{i_{0}} \cap \dotsb \cap H_{i_{s}} \, \mid \, i_{1} \leq \dotsb \leq i_{s},\;1 \leq s \leq m \}.
$$
Notice that $L(\cA)$ is partially ordered by reverse inclusion. For $X \in L(\cA)$ we define the {\em localized arrangement} of $\cA$ at $X$ to be $\cA_{X} := \{H \in \cA\, \mid\, X \subset H \}$.

\begin{definition} \rm
With the above notation, the {\em module} $\Der(-\log \cA )$ {\em of logarithmic derivations}
of $\cA $ is the set of $R$-linear derivations $\theta \in \Der_k(R)$ such that $\theta(f_{\cA }) \subset  (f_{\cA})$.
 A derivation
$\theta \in \Der(-\log \cA )$ has degree $d$ if $\theta =\sum _{i=0}^n \theta _i \partial_{x_{i}}$ with $\theta _i\in R_d$. The Euler derivation
$\theta _E:=\sum _{i=0}^n x_i\partial_{x_{i}}$ generates a free
submodule $R\cdot \theta _E$ of $\Der(-\log \cA )$ of rank one. The quotient of $\Der(-\log \cA )$ by $R\cdot \theta _E$ is denoted by $\Der(-\log \cA )_0$. The {\em sheaf of logarithmic vector fields} $\cT_{\PP^n}(-\log D_{\cA} )$, which we will also denote by $\cT_\cA$, is defined as
the sheafification of $\Der(-\log \cA )_0$.
\end{definition}

\begin{definition} \rm
A hyperplane arrangement $\cA $ in
$\PP^n$ is {\em free} if $\Der(-\log \cA )$ is a free $R$-module of rank $n+1$. In this
case, the degrees $1,d_1, \ldots, d_{n}$ of the generators $\theta_E, \theta_1, \ldots, \theta_{n}$ of $\Der(-\log \cA )$ are called the {\em exponents of the arrangement}.
\end{definition}

By \cite[Proposition 2.4]{Terao}, we have 
$$\Der(-\log \cA )=R\cdot \theta _E \oplus \syz (J_{\cA}), $$
where $J_{\cA} := J_{f_\cA}$ is the Jacobian ideal of $f_{\cA}$, and $\syz (J_{\cA})$ denotes the module of
syzygies on $J_{\cA}$, i.e. the polynomial relations on the generators of $J_{\cA}$. Notice that $\syz (J_{\cA})$ is isomorphic to $\Der(-\log \cA )_0$.

\begin{theorem} The hyperplane arrangement $\cA $ is free if and only if there exist $n+1$ logarithmic derivations
$$\theta _i =\sum _{j=0}^nf_{ij}\partial _{x_{i}} \in \Der(-\log \cA )$$
such that $det([f_{ij} ]) = c\cdot f_{A}$ for some $c\ne  0.$
\end{theorem}

\begin{proof} See \cite[Theorem 1.8]{S}.
\end{proof}

\begin{example} \label{ex1} \rm
(i) The {\em boolean arrangement} $\cA =\{x_0,x_1,\hdots, x_n\}$ is free with exponents $(1, \hdots, 1)$. Its Jacobian ideal $J_{\cA}$
has the following free $R$-resolution
$$ 0\longrightarrow R(-1)^n\longrightarrow R^{n+1} \longrightarrow J_{\cA}(n)\longrightarrow 0.$$

(ii) The {\em braid arrangement} $\cA _3$ in $\PP^{3}$ with defining equation $f_{\cA} = (x_0-x_1)(x_0-x_2)(x_0-x_3)(x_1-x_2)(x_1-x_3)(x_2-x_3)$ is free with exponents $(0,1,2,3)$. The Jacobian ideal $J_{\cA _3}$ has the free $R$-resolution:
$$0\longrightarrow R\oplus R(-2)\oplus R(-3)\longrightarrow R^4\longrightarrow J_{\cA _3}(5)\longrightarrow 0.$$

(iii) The Jacobian ideal $J_{\cA}$  of the line arrangement $\cA$ in $\PP ^2$ with defining equation $f_{\cA} = x_0x_1x_2(x_0+x_1+x_2)$ has a minimal free $R$-resolution:
$$
0\longrightarrow R(-3)\longrightarrow R(-2)^3\longrightarrow R^3\longrightarrow J_{\cA }(3)\longrightarrow 0.
$$
Therefore, $\cA $ is not free.
\end{example}

Generalizing Example \ref{ex1} (ii), the braid arrangement $\cA_n$ in $\PP^n$ is defined by the equation $f_{\cA_n}=\prod_{0\le i<j\le n}(x_i-x_j)$. In particular, $\cA_{n}$ is free with exponents $(0,1, \hdots ,n)$. Moreover, let $\Gamma = (V,E)$ be a graph with set of vertices $V = \{0, \dotsc, n\}$  and $E$ its set of edges. We define the {\em graphic arrangement} $\cA_{\Gamma}$ associated to $\Gamma$, as the subarrangement of $\cA_{n}$ with equation
$$
f_{\Gamma}=\prod_{(i,j) \in E}(x_{i} - x_{j}).
$$
The freeness of graphic arrangements is characterized using the chordality of the associated graph. More precisely, we have the following definition.
\begin{definition}\rm \label{Definition:chordal_complete star} A graph is called {\em chordal} if  any cycle has a {\em chord}, that is,  an edge not in the cycle which connects two vertices.

In particular, {\em a complete star} is a graph in which all of its vertices are only connected to a fixed one. Notice that, not having cycles of length greater or equal than three, a complete star is chordal.
\end{definition}
\begin{remark}\rm\label{Remark:graphic arrangement}
In terms of the defining equations of the hyperplanes of a graphic arrangement $\cA_{\Gamma}$, we can characterize them as follows:
\begin{itemize}
    \item[(i)] $\cA_{\Gamma}$ is chordal if, for any set $\{i_1,\ldots,i_k\}$, with $4 \leq k \leq n$, such that 
    $$
    (x_{i_1} - x_{i_k}) \prod_{j=1}^{k-1} (x_{i_j}- x_{i_{j+1}}) \mid f_{\cA_{\Gamma}},
    $$ 
    we have at least another hyperplane in the arrangement defined by
    $
    x_{i_h} - x_{i_s},
    $
    with $h<s$,  $\{h,s\}\neq \{1,k\}$ and $s-h\geq 2$;
    \item[(ii)] $\cA_{\Gamma}$ corresponds to a complete star $\Gamma$ if it is defined, up to change of coordinates, by
    $$
    f_{\cA_{\Gamma}} = \prod_{i=1}^m (x_0 - x_i), \:\: \mbox{ with } m<n.
    $$
\end{itemize}
\end{remark}

\begin{example} \rm
(i) The graphic arrangement $\cA$ defined by
$$f_\cA = (x_0-x_1)(x_1-x_2)(x_2-x_3)(x_3-x_4)(x_0-x_4)$$ is not chordal. Indeed, 
the graph $\Gamma_{\cA}$ associated to $\cA$ is a cycle $(0,1,2,3,4)$ of length 5 having no chord (see Figure \ref{figpentagon}).  
On the other hand, adding two hyperplanes and considering
$$f_{\cA'} = (x_0-x_1)(x_1-x_2)(x_2-x_3)(x_3-x_4)(x_0-x_4)(x_1-x_3)(x_1-x_4)$$
makes $\cA'$ a chordal graphic arrangement. Indeed, in the associated graphic $\Gamma_{\cA'}$, any cycle of length $4$, $((0,1,3,4) \, \text{and} \, (1,2,3,4)$, and of length $5$, $(0,1,2,3,4)$, has a chord (see Figure \ref{figchordal}). 

(ii) The graphic arrangement $\cA''$ obtained from $\cA'$ by removing the hyperplanes defined by the equations $x_2 - x_3=0$, $x_0 - x_4=0$ and $x_3-x_4=0$ is a complete star graphic arrangement. Indeed, in the associated graph $\Gamma_{\cA''}$ all the vertices are only connected to the vertex $1$ (see Figure \ref{figcompletestar}).
\begin{figure}[!htb]
   \begin{minipage}{0.30\textwidth}
     \centering
     \includegraphics[scale=0.35]{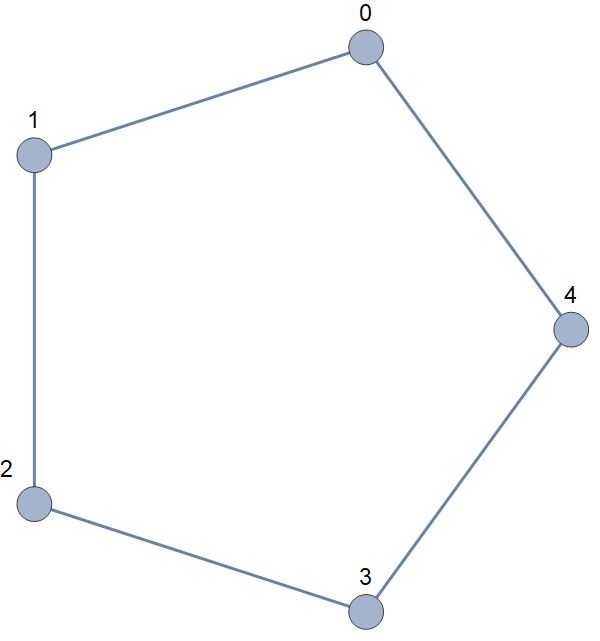}
\caption{$\Gamma_{\cA}$}\label{figpentagon}
   \end{minipage}
   \hfill
   \begin{minipage}{0.30\textwidth}
     \centering
     \includegraphics[scale=0.35]{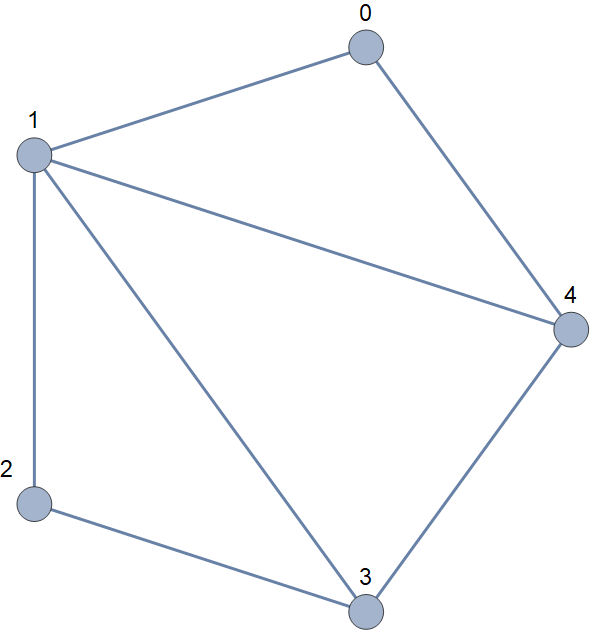}
\caption{$\Gamma_{\cA'}$}\label{figchordal}
   \end{minipage}
   \hfill
   \begin{minipage}{0.30\textwidth}
     \centering
     \includegraphics[scale=0.35]{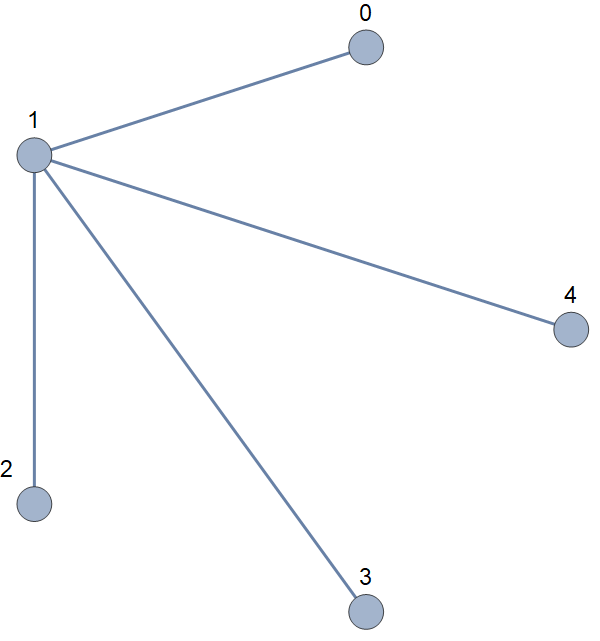}
\caption{$\Gamma_{\cA''}$}\label{figcompletestar}
   \end{minipage}
\end{figure}
\end{example}

\begin{proposition}\label{Prop:freenes graphic arrangement}
 Let $\cA_{\Gamma}$ be the graphic arrangement associated to a graph $\Gamma$. Then, $\cA_{\Gamma}$ is free if and only if $\Gamma$ is chordal.
\end{proposition}
\begin{proof} See \cite[Theorem 3.3]{Edelman-Reiner}.
\end{proof}

Next we define the product of two hyperplane arrangements and characterize its freeness.

\begin{definition}\rm
Let $\cA_{1}$ and $\cA_{2}$ be two hyperplane arrangements in $\PP^{n}$ and $\PP^{m}$, respectively. Let $f_{\cA_{1}}\in k[x_{0}, \dotsc, x_{n}]$ and $f_{\cA_{2}} \in k[y_{0}, \dotsc, y_{m}]$ be their equations. We define the {\em product arrangement} $\cA_{1} \times \cA_{2}$ as the hyperplane arrangement in $\PP^{n+m+1}$ with equation  $f_{\cA_{1}}f_{\cA_{2}} \in k[x_{0}, \dotsc, x_{n}, y_{0}, \dotsc, y_{m}]$.
\end{definition}

\begin{proposition}\label{prop:freeness of product}
Let $\cA_{1}$ and $\cA_{2}$ be two hyperplane arrangements in $\PP^{n}$ and $\PP^{m}$, respectively. The product arrangement $\cA_{1} \times \cA_{2}$ is free if and only if both $\cA_{1}$ and $\cA_{2}$ are free.
\end{proposition}
\begin{proof}
See \cite[Proposition 4.28]{Orlik-Terao}.
\end{proof}

We end this preliminary section with a lemma which will play an important role later.

\begin{lemma}\label{section} Let $\cA =\{ H_1, \hdots, H_r \}$, $r\ge 2$, be a hyperplane arrangement in $\PP^n$ with equation $f _{\cA }$. Let $\cA '=\cA \bigcup \{ H_{r+1}, \hdots, H_{r+s} \}$, $s\ge 1$, be another hyperplane arrangement in $\PP^n$ with equation $f_{\cA'}=f_{\cA }\prod_{j=1}^sh_{r+j}$ being $h_{i}$ the linear equation of $H_i$. In particular, $(\bigcup _{i=1}^rH_i)\bigcap(\bigcup _{j=1}^sH_{r+j})$ has codimension 2 in $\PP^n$. Set $g_s=\prod_{j=1}^sh_{r+j}$. Then, there is an injection
$$0\longrightarrow \cT_{\cA} (-s) \stackrel{g_s }{\longrightarrow } \cT_{\cA '}.$$
\end{lemma}

\begin{proof}
Let us consider an open subset $U \subset \mathbb{P}^n$ and the local description of the logarithmic tangent sheaves involved, see for example \cite{Dolgachev} for more details. Denote also by $f_\cA$ and $f_{\cA'}$ the equations which define locally the hyperplanes. Consider a derivation $\theta \in \left(\cT_{\cA} (-s)\right)_{\mid U}$, i.e. $\theta(f_\cA) \in (f_\cA)$. We have that $g_s\theta(f_{\cA'}) = g_s \theta(f_{\cA}) g_s + g_s f_{\cA} \theta(g_s) \in (g_s f_{\cA}) = (f_{\cA'})$. Therefore $g_s \theta \in \left(\cT_{\cA'} \right)_{\mid U}$ and this concludes the proof.
\end{proof}


\section{The local freeness of hypertetrahedral arrangements}
\label{local freeness of Hypertetrahedral arrangements}

This section is entirely devoted to study the local freeness of hypertetrahedral arrangements. Let us introduce their definition and fix the notation we will use in the sequel. We denote by $\mathcal{L}_{i,j}\subset \PP^n$ the codimension 2 linear subspace passing through the vertices $e_0, \hdots, \widehat{e_{i}}, \hdots, \widehat{e_{j}}, \hdots, e_{n}$ with $i < j$.

\begin{definition}\rm A {\em hypertetrahedral arrangement} is a hyperplane arrangement $\cA $ such that any of its hyperplanes passes through at least one linear subspace $\mathcal{L}_{i,j}\subset \PP^n$, $0\le i < j \le n$. In particular, when $n=2$, we will call it a \textit{triangular arrangement} of lines in $\PP^2$. The hypertetrahedral arrangement $\cA $ is complete if it contains the hyperplanes $L_i:= \{ x_i=0\}$, \; $i=0,\hdots ,n$.
\end{definition}

 We will denote $$\cH(s_{i,j})_{0 \leq i < j \leq n}$$ (or, simply   $\cH(s_{i,j})$)
 the set of all complete hypertetrahedral arrangements $\cA $ in $\PP^n$ with $s_{i,j} + 2$ different hyperplanes passing through $\mathcal{L}_{i,j}$. We will assume that, for all $0 \leq i < j \leq n$, $s_{i,j} \geq 1$ and without loss of generality, we can always assume that $s_{n-1,n} \geq s_{i,j}$, for all $0 \leq i < j \leq n$.

\vspace{3mm}

For any $0 \leq i < j \leq n$, we denote by
$$L_{i,j}^{r}=a_i^{(r;i,j)}x_i+a_j^{(r;i,j)}x_j, \quad \mbox{for} \quad 1 \leq r \leq s_{i,j},$$
where we always assume that the coefficients $a_i^{(r;i,j)}$ and $a_j^{(r;i,j)}$ are different from zero. We  call them the {\em inner hyperplanes} of $\cA$.
We denote by  $L_{i,j}= \{ L_{i,j}^{r}=0 \}_{1 \leq r \leq s_{i,j}}$ the family of inner hyperplanes passing through $\mathcal{L}_{i,j}$. According to this notation our arrangement is given by
$$\cA=\{x_0,x_1,\hdots ,x_n,L_{0,1}^{1}, \hdots,L_{0,1}^{s_{0,1}}, \hdots,L_{n-1,n}^{1}, \hdots,L_{n-1,n}^{s_{n-1,n}}  \}.$$

 Observe that any of these arrangements $\cA \in \cH(s_{i,j})$ has $n+1+\sum _{0\le i < j \leq n} s_{i,j}$ hyperplanes.

\begin{example} \rm

Fix an integer $a\ge 2$. The {\em Fermat arrangement} $\cA$ in $\PP^n$, defined by $\prod _{0\le i<j\le n} (x_i^a-x_j^a)$, is free with exponents $(1, a+1, 2a+1,  \hdots  , (n-1)a+1, na+1)$. In this case, for all $0\le i<j\le n$ and $1 \le r \le a$, we have $L_{i,j}^r=x_i-\eta ^{r-1}x_j$ where $\eta$ is a  primitive $a$-th root of 1. Therefore, we have: $a_i^{(r;i,j)}=1$ and $a_j^{(r;i,j)}=- \eta^{r-1}$.
\end{example}

Despite the above examples, complete hypertetrahedral arrangements are in general not free and the next goal is to characterize whether they are locally free.
For any hyperplane arrangement in $\PP^{n}$, the local freeness was characterized by Musta\c{t}\u{a} and Schenck in \cite[Theorem 3.3]{Mustata-Schenck}. We specialize this criterion for complete hypertetrahedral arrangements.
We fix $\cA \in \cH(s_{i,j})$. For any inner hyperplane $H \in \cA$ with equation $x_{i}-\lambda x_{j}$, with $\lambda\neq 0$, we define $\supp(H) := \{x_{i},x_{j}\}$ the support of $H$. Furthermore, if $H$ is the coordinate hyperplane with equation $x_i$, we define $\supp(H) := \{x_{i}\}$. Similarly, given $\cA'$ a subarrangement of $\cA$, we define $\supp(\cA') := \bigcup_{H \in \cA'} \supp(H)$. In particular, for $X \in L(\cA)$ we define $\supp(X) := \supp(\cA_{X})$ the support of $X$. Finally, for a subarrangement $\cA'$ of $\cA$, we denote by $n-k-1$ the cardinality of the linear system ${\mathcal L}:=\langle \{ x_{j} \, \mid \, x_{j} \notin \supp(\cA')\} \rangle$ and by $\;  \pi \,: \, \PP^{n} \to \PP^{k}$ the projection associated to 
${\mathcal L}$.

We begin associating to each $X \in L(\cA)$ a pair of arrangements $(\cA_{W_{X}}, \cA_{\Gamma_{X}})$,  where $\cA_{W_{X}}$ is a complete hypertetrahedral arrangement and $\cA_{\Gamma_{X}}$ is a graphic arrangement.

\begin{remark}\rm  Let $\cA'$ be a subarrangement of $\cA$. Assume that
$\supp(\cA') = \{x_{i_{0}},\hdots, x_{i_{k}}\}$ with $k < n$. Then
$(\cT_\cA')_{\mid \PP^k}$ is isomorphic to $\cT_{\pi(\cA')} \oplus \mathcal{O}_{\PP^k}^{n-k-1}$.  Hence  $\cA'$ is free if and only if $\pi(\cA')$ is free.
\end{remark}

\begin{definition}\rm
Let $X\in L(\cA)$ be a linear subspace. Define $W_{X}$ as the intersection, of minimal dimension, of coordinate hyperplanes containing $X$. If $X$ is not contained in any coordinate hyperplane, then we set $W_{X}:=\emptyset$ and we say that {\em $X$ is inner}.
\end{definition}

Since $W_{X} \in L(\cA)$, we can consider the localized arrangement $\cA_{W_{X}}$. Assuming that $W_{X}$ is the intersection of $j$ coordinate hyperplanes, we have that $\pi(\cA_{W_{X}})$ is a complete hypertetrahedral arrangement in $\PP^{j-1}$. The following lemma shows that it is a component of $\pi(\cA_{X})$.

\begin{lemma}\label{Lemma:Direct sum decomposition}
$\pi(\cA_{X}) \cong \pi(\cA_{W_{X}}) \times \pi(\cA_{X} - \cA_{W_{X}})$.
\end{lemma}
\begin{proof}
We may assume that $W_{X} \neq \emptyset$ and $X \subsetneq W_{X}$, otherwise there is nothing to prove.
Let $j$ be the codimension of $W_{X}$. Without loss of generality we assume that $\supp(X) = \{x_{0}, \dotsc, x_{k}\}$ with $j\leq k$ and $W_{X} = H_{0} \cap \dotsb \cap H_{j-1}$. Since $\cA_{X} = \cA_{W_{X}} \bigcup (\cA_{X} - \cA_{W_{X}})$, it is enough to prove that if $H \in \cA_{X} - \cA_{W_{X}}$, then $\supp(H) \subset \{x_{j}, \dotsc, x_{k}\}$. Notice that $\supp(H)\nsubseteq \{x_{0}, \dotsc, x_{j-1}\}$, otherwise $H \in \cA_{W_{X}}$. By contradiction, let us assume that $H$ has equation $x_{l} - \lambda x_{m}$ with $0 \leq l < j \leq m \leq k$. There exists $p = (a_{0}: \dotsb :a_{k}:a_{k+1}: \dotsb :a_{n}) \in X$, such that $a_{0} = \dotsb = a_{j-1} = 0$ and $a_{m} \neq 0$. Otherwise $X$ is contained in $H_{m}$ and $W_{X}$ is not of minimal dimension. Since $p \in H$, then $\lambda = 0$ and $H = H_{l}$ which is a contradiction.
\end{proof}

\begin{remark} \rm \label{Remark:Inner intersection}
Let $Y$ be the intersection of all hyperplanes in $\cA_{X} - \cA_{W_{X}}$.
Then $X = Y \cap W_{X}$ and $\pi(Y)$ is inner in $L(\pi(\cA_{X} - \cA_{W_{X}}))$. Therefore $\cA_{X} - \cA_{W_{X}}$ contains at most one hyperplane passing through each $\mathcal{L}_{i,j}$.
\end{remark}

Now we attach a graphic arrangement $\cA_{\Gamma_{X}}$ to $\cA_{X}$ using the subarrangement $\cA_{X}-\cA_{W_{X}}$. Next, we prove that the pair $(\cA_{W_{X}}, \cA_{\Gamma_{X}})$ determines completely the freeness of $\cA_{X}$.

\begin{definition} \rm  Let $X \in L(\cA)$ be a linear subspace such that $X \subsetneq W_{X}$ and assume that $\supp(\cA_{X} - \cA_{W_{X}}) = \{x_{i_{0}}, \dotsc, x_{i_{k}}\}$. We define $\Gamma_{X} = (V_{X},E_{X})$ the graph associated to $X$ with vertices $V_{X} = \{i_0,\hdots,i_k\}$ and $E_{X}$ the set of edges $(i_l,i_m)$ such that there is $H \in \cA_{X} - \cA_{W}$ with $\supp(H) = \{x_{i_{l}},x_{i_{m}}\}$. We define $\cA_{\Gamma_{X}}$ the graphic arrangement associated to $\Gamma_{X}$.

If $X = W_{X}$ we set $\Gamma_{X} := \emptyset$.
\end{definition}

\begin{theorem}\label{Theorem: Main Locally Free} Let $X \in L(\cA)$ be a linear subspace.
\begin{itemize}
\item[(i)] $\pi(\cA_{X}) \cong \pi(\cA_{W_{X}}) \times \cA_{\Gamma_{X}}$.
\item[(ii)] $\cA_{X}$ is free if and only if $\pi(\cA_{W_{X}})$ is free and $\Gamma_{X}$ is chordal.
\end{itemize}
\end{theorem}
\begin{proof}
We assume that $X \subsetneq W_{X}$, otherwise the result follows directly from \cite[Theorem 2.3]{Mustata-Schenck}.

(i) Without loss of generality, we  suppose that $\supp(\cA_{X} - \cA_{W_{X}}) = \{x_{0}, \dotsc, x_{k}\}$. By Lemma \ref{Lemma:Direct sum decomposition}, it is enough to see that $\pi(\cA_{X} - \cA_{W_{X}})$ is isomorphic to $\cA_{\Gamma_{X}}$. Let $\pi(Y)$ be as in Remark \ref{Remark:Inner intersection}. Since $\pi(Y)$ is inner in $L(\pi(\cA_{X} - \cA_{W_{X}}))$, there exists $p=(a_{0}: \dotsb :a_{k}) \in \pi(Y)$ with $a_{i} \neq 0$, for all $i$. We consider the projectivity $\phi$ that
fixes the coordinate points and sends the point $p$ to the unit point $(1:\cdots:1)$. We claim that $\phi$ is the desired isomorphism. 
Indeed, we have at most one hyperplane $H_{ij}$ in $\pi(\cA_{X} - \cA_{W_{X}})$ containing $\cL_{ij}\subset \PP^{k}$. Since $\phi(H_{ij})$ passes through the unit point, then up to scalar multiplication $\phi(H_{ij})$ is defined by $x_{i}-x_{j}$.

(ii) The freeness of $\cA_{X}$ is equivalent to the freeness of $\pi(\cA_{X})$. Then, by (i) $\cA_X$ is free if and only if both $\pi(\cA_{W_{X}})$ and $\cA_{\Gamma_{X}}$ are free. Finally, by Proposition \ref{Prop:freenes graphic arrangement}, $\cA_{\Gamma_{X}}$ is free if and only if $\Gamma_{X}$ is chordal, which completes the proof.
\end{proof}

\begin{example}\rm
Let $\cA$ be the complete hypertetrahedral arrangement in $\PP^{5}$ with equation
$$f_{\cA}=x_{0}x_{1}x_{2}x_{3}x_{4}x_{5}
(x_{0}-\frac{1}{2}x_{1})(x_{0}-\frac{1}{3}x_{1})
(x_{0}+x_{2})
(x_{0}-x_{3})
(x_{0}+\frac{1}{3}x_{4})
(x_{0}-\frac{1}{2}x_{5})
$$
$$
(x_{1}+x_{2})(x_{1}+3x_{3})(x_{1}+x_{4})(x_{1}+x_{5})
(x_{2}^{2}-x_{3}^{2})(x_{2}^{2}-x_{4}^{2})(x_{2}^{2}-x_{5}^{2})
$$
$$
(x_{3}^{2}-x_{4}^{2})(x_{3}-\frac{1}{3}x_{4})(x_{3}+2x_{4})
(x_{3}^{2}-x_{5}^{2})(x_{4}^{2}-x_{5}^{2}).
$$
Consider $X_{1}=\{(1:2:-1:1:-2:2)\}$, $X_{2}:=\{x_{2}=x_{3}=x_{4}=x_{5}=0\}$ and $X_{3}=\{(1:3:0:-1:-3:0)\}$ three elements of $L(\cA)$. The localized arrangements at $X_{i}$ have equations:
$$
f_{X_{1}}=(x_{0}-\frac{1}{2}x_{1})(x_{0}+x_{2})(x_{0}-x_{3})(x_{0}-\frac{1}{2}x_{5})(x_{1}+x_{4})(x_{2}+x_{3})(x_{4}+x_{5}).
$$
$$
f_{X_{2}}=x_{2}x_{3}x_{4}x_{5}(x_{2}^{2}-x_{3}^{2})(x_{2}^{2}-x_{4}^{2})(x_{2}^{2}-x_{5}^{2})
(x_{3}^{2}-x_{4}^{2})(x_{3}-\frac{1}{3}x_{4})(x_{3}+2x_{4})
(x_{3}^{2}-x_{5}^{2})(x_{4}^{2}-x_{5}^{2}).
$$
$$
f_{X_{3}}=x_{2}x_{5}(x_{0}-\frac{1}{3}x_{1})(x_{0}+\frac{1}{3}x_{4})(x_{1}+3x_{3})(x_{3}-\frac{1}{3}x_{4})(x_{2}^{2}-x_{5}^{2}).
$$
$X_{1}$ is inner with the graph in Figure \ref{fig1}.
Notice that the cycle $(0,1,4,5)$ has no chords, so $\Gamma_{X_{1}}$ is not chordal and hence, $\cA_{X_{1}}$ is not free.
On the other hand, $X_{2}$ is an intersection of coordinate hyperplanes and by Proposition \ref{sharp}, $\cA_{X_{2}}$ is free. Finally, $W_{X_{3}}=\{x_{2}=x_{5}=0\}$ and $\Gamma_{X_{3}}$ is the graph in Figure \ref{fig2}. Since $\pi(\cA_{W_{X_{3}}})$ is an arrangement in $\PP^{1}$, it is free. However, $\cA_{X_{3}}$ is not free because the cycle $(0,1,3,4)$ in $\Gamma_{X_{3}}$ has no chord.

\begin{figure}[!htb]
   \begin{minipage}{0.48\textwidth}
     \centering
     \includegraphics[scale=0.35]{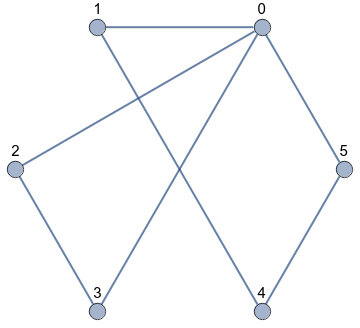}
\caption{$\Gamma_{X_{1}}$}\label{fig1}
   \end{minipage}\hfill
   \begin{minipage}{0.48\textwidth}
     \centering
     \includegraphics[scale=0.35]{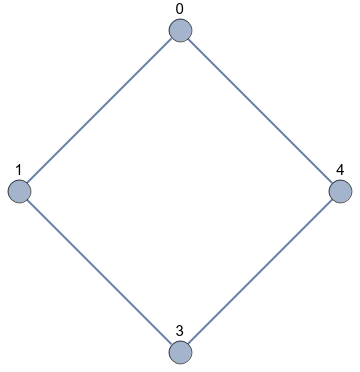}
\caption{$\Gamma_{X_{3}}$}\label{fig2}
   \end{minipage}
\end{figure}
\end{example}

As a direct consequence of Theorem \ref{Theorem: Main Locally Free} we obtain a characterization of the local freeness of complete hypertetrahedral arrangements. Namely, we express it  using only smaller dimensional complete hypertetrahedral arrangements and a graph theory property.

\begin{corollary} $\cA$ is locally free if and only if for each $X \in L(\cA)$, the complete hypertetrahedral arrangement $\cA_{W_{X}}$ is free and $\Gamma_{X}$ is chordal.
\end{corollary}

We end this section with two  applications of Theorem \ref{Theorem: Main Locally Free}.

\begin{proposition} Let $X \in L(\cA)$ be inner with $\supp(X) = \{x_{i_{0}},\hdots, x_{i_{k}}\}$. Assume that there are $H_{1},\hdots, H_{k} \in L(\cA_{X})$ linearly independent hyperplanes such that $\bigcap_{l=1}^{k} \supp(H_{l}) = \{x_{i_{j}}\}$ for some  $x_{i_{j}} \in \supp(X)$. Then $\cA_{X}$ is free.
\end{proposition}

\begin{proof} Without loss of generality, we can assume that $\supp(X) = \{x_{0},\hdots, x_{k}\}$ and $\bigcap_{l=1}^{k} \supp(H_{l}) = \{x_{0}\}$.
By Theorem \ref{Theorem: Main Locally Free}, $\cA_{X}$ is free if and only if $\Gamma_{X}$ is chordal. The hypothesis implies that we can consider $H_{1},\hdots, H_{k}$ such that $\supp(H_{i}) = \{x_{0},x_{i}\}$. Hence
$\{(0,1), \hdots, (0,k)\} \subset E_{X}$, which is a complete star (see Definition \ref{Definition:chordal_complete star}) and then the graph is chordal.
\end{proof}

\begin{proposition} Let $n \geq 3$ and $\cA \in \cH(s_{i,j})$. If $\cA$ is general, then $\cA$ is not locally free.
\end{proposition}
\begin{proof} We proceed by induction on $n$. For $n = 3$, we set $X = \{(0:0:0:1)\}$. Then $\pi(\cA_{X}) = \pi(\cA_{W_{X}})\in \cH(s_{0,1},s_{0,2},s_{1,2})$ is a complete triangular arrangement in $\PP^{2}$, which is also general. By \cite[Proposition 2.2]{MV}, we have the following exact sequence:
$$
0 \longrightarrow \mathcal{T}_{\pi(\cA_{X})} \longrightarrow \bigoplus_{0 \leq i < j \leq 2} \mathcal{O}(-s_{i,j} -1) \longrightarrow I_{G}(-1) \longrightarrow 0, $$
where $G$ is the set of inner triple points of $\pi(\cA_{X})$. Since $\pi(\cA_{X})$ is also general,  $G = \emptyset$. If  $\mathcal{T}_{\pi({\cA}_{X})}$ is free, then necessarily some $s_{i,j}$ equals to $0$, which is a contradiction. Therefore by Theorem \ref{Theorem: Main Locally Free}, $\cA$ is not locally free.

Now we assume that the result holds in $\PP^{n-1}$ for $n \geq 4$. We localize $\cA$ at $X=\{(0: \hdots :0:1)\}$, as before $\pi(\cA_{X}) = \pi(\cA_{W_{X}})$ is a generic complete hypertetrahedral arrangement in $\PP^{n-1}$ with $s_{i,j} \geq 1$ for $0\leq i < j \leq n-1$. By induction, $\pi(\cA_{X})$ is not locally free, hence not free and the result follows from Theorem \ref{Theorem: Main Locally Free}.
\end{proof}

We remark that not all locally free complete hypertetrahedral arrangements are free, as the following example shows.

\begin{example} \rm Let $\cA$ be the complete hypertetrahedral arrangement in $\PP^{3}$ with defining equation
$$f_{\cA} = x_{0}x_{1}x_{2}x_{3}(x_{0}-x_{1})(x_{0}-2x_{1})(x_{0}-x_{2})
(x_{0}-x_{3})(x_{1}-x_{2})(x_{1}-x_{3})(x_{2}^{2}-x_{3}^{2}).$$
Applying Theorem \ref{Theorem: Main Locally Free}, it is straightforward to check that $\cA$ is locally free. Using the software Macaulay2 (\cite{M2}), we compute a minimal free resolution of $J_{\cA}$:
$$0 \to R(-7) \to R(-5) \oplus R(-6)^{3} \to R(-4)^{3} \oplus R(-5)^{3} \to R^{4} \to J_{\cA}(11) \to 0,$$
which shows that $\cA$ is not free.
\end{example}


\section{Jacobian ideal of hypertetrahedral arrangements}
\label{JacobianIdeal}

\vspace{0.3cm}
The goal of this section is to establish lower and upper bounds for the initial degree of the syzygy module associated to any complete hypertetrahedral arrangement in terms of the values $s_{i,j}$. Given $\cA \in \mathcal{H}(s_ {i,j})$, we define $0 \leq i_0 \leq n$ to be an integer such that
$$\sum _{\substack{j=0 \\ j\ne i_0}}^ns_{i_0,j} = \min\{\sum _{\substack{j=0\\ j\ne i}}^ns_{i,j}, \;\; 0\le i \le n\}.$$
We have the following result. 

\begin{proposition}
\label{Dimcan}
 With the above notation, the Jacobian ideal $J_{\cA}$ of  any  complete hypertetrahedral arrangement $\cA \in \mathcal{H}(s_ {i,j})$  has a syzygy of degree $\sum _{\substack{j=0 \\ j \neq i_0}}^{n}s_{i_0,j}+1$. In particular,
$$\indeg(\syz(J_{\cA}))\le \sum _{\substack{j=0 \\ j \neq i_0}}^{n} s_{i_0,j}+1.$$
\end{proposition}

\begin{proof} To simplify the notation, we assume that $i_0 = 0$. Set $d=n+1+\sum _{0 < i \le n}s_{0,i}$ and take $$\cA =\{x_0,x_1,\hdots ,x_n,L_{0,1}^{1}, \hdots,L_{0,1}^{s_{0,1}}, \hdots,L_{n-1,n}^{1}, \hdots,L_{n-1,n}^{s_{n-1,n}}  \}.$$  We write the defining equation of $\cA$ as $f_{\cA}=gh$, where $g=x_1\cdots x_{n}\prod _{0 < i<j\le n}\prod _{1\le t\le s_{i,j}}L_{i,j}^t$ is a polynomial of degree $n+\sum _{0 <  i<j\le n}s_{i,j}$ and $h=x_0\prod_{0 < i\le n}\prod _{1\le t\le s_{0,i}}L_{0,i}^t$ is a polynomial of degree $1+\sum _{0 < i\le n}s_{0,i}$. To simplify, we write $h=\prod _{i=1}^{\alpha }L_i$ where $\alpha =1+\sum _{0 < i\le n}s_{0,i}$ and we denote by $a_{L_{i}}$ the coefficient of $x_0$ in $L_i$. On has
$$\partial_{x_{0}} f_{\cA}=g \partial_{x_{0}} h =gh\sum _{i=1}^{\alpha }\frac{a_{L_i}}{L_{i}}=f_{\cA }\sum _{i=1}^{\alpha }\frac{a_{L_i}}{L_{i}}=
f_{\cA} \frac{P}{h},$$
where $P$ is a polynomial of degree  $\sum _{0 < i\le n}s_{0,i}$ such that $\gcd(P,h)=1$. Therefore we obtain
 $$dh\partial_{x_{0}} f_{\cA}=dPf_{\cA}=P(\sum _{i=0}^nx_i \partial_{x_{i}} f_{\cA})$$
 or equivalently,
 $$(x_0P-dh)\partial_{x_{0}} f_{\cA}+P(\sum _{i=1}^n x_i \partial_{x_{i}} f_{\cA})=0$$
 as we wanted to prove. 
\end{proof}

The following result shows that the above upper bound for the initial degree of the first syzygy module of the Jacobian ideal is sharp for complete triangular arrangements. 

\begin{proposition}
Let $\cA \in \cH(s_{i,j})$ be a complete general triangular arrangement with $s_{0,1} \leq s_{0,2} \leq s_{0,3}$. Then $\indeg(J_{\cA}) = s_{0,1}+s_{0,2}+1$. 
\end{proposition}

\begin{proof} To simplify the notation we set $s_{0,1} = s_{1}, s_{0,2} = s_{2}$ and $s_{1,2} = s_{3}$. We proceed by induction on $s_{1}$. 
The initial case $s_{1} = 0$ follows directly from Theorem \ref{triang}(i). 

Let $s_{1} \geq 1$ and we assume that the result is true in $\cH(s_{1}-1,s_{2},s_{3})$, for any integers $s_{1}-1 \leq s_{2} \leq s_{3}$. Let $\cA \in \cH(s_1,s_2,s_3)$ be a complete general triangular arrangement and let $l$ be any non coordinate line passing through $L_{0,1}$. Since $l$ has no multiple points except from the vertex, we dualize the exact sequence $0 \to \cT_{\cA} \to \cT_{\cA-l} \to 
\mathcal{O}_{l}(-s_1) \to 0$ (see \cite[Proposition 5.1 and 5.2]{FV}) and then we obtain 
$$0 \to (\cT_{\cA-l})^{\vee} \to (\cT_{\cA})^{\vee} \to 
\mathcal{O}_{l}(s_1+1) \to 0.$$
Given that $\cT_{\cA-l}$ and $\cT_{\cA}$ have both rank $2$, we get equivalently 
$$0 \to \cT_{\cA-l}(s_1+s_2+s_3+1) \to \cT_{\cA}(s_1+s_2+s_3+2) \to 
\mathcal{O}_{l}(s_1+1) \to 0.$$
After applying $\otimes \mathcal{O}_{\PP^{2}}(-s_3-2)$, we obtain
$$0 \to \cT_{\cA-l}(s_1 + s_2 -1) \to \cT_{\cA}(s_1 + s_2) \to \mathcal{O}_{l}(s_1-s_3-1) \to 0.$$ 
Notice that $\mbox{H}^{0}(\mathcal{O}_{l}(s_1 - s_3 -1)) = 0$, indeed 
$s_{1} \leq s_{3}$. Since $\cA\setminus \{l\}$ is also general, by induction we have $\mbox{H}^{0}(\cT_{\cA-l}(s_1+s_2-1)) = 0$, which implies that $\mbox{H}^{0}(\cT_{\cA}(s_1+s_2)) = 0$. Therefore, $\indeg(J_{\cA}) \geq s_1+s_2+1$ and by Proposition \ref{Dimcan}, the result follows. 
\end{proof}

Let us see a couple of examples that illustrate Proposition \ref{Dimcan}. 
\begin{example} \rm \label{extendedFermat}  (i) With the above notation, assume $s_1+s_2 \leq s_3$ and set $\rho$ a primitive root of unity of order $s_3$. The complete triangular arrangement $\cA \in \cH(s_1,s_2,s_3)$ with defining equation
$$x_0x_1x_2\prod _{j=0}^{s_1-1}(x_0-{\rho}^{j}x_1)\prod _{j=0}^{s_2-1}(x_0-\rho^j x_2) \prod _{j=0}^{s_3-1}(x_1-\rho^j x_2)$$
is free with exponents $(1,s_1+s_2+1,s_3)$ (see \cite[Remark 2.7]{MV}).

(ii) The {\em extended Fermat arrangement} has defining equation $x_0x_1\cdots x_n\prod _{0\le i<j \le n}(x_i^a-x_j ^a)$ and it
is free with exponents $(1, a+1, 2a+1, \hdots , na+1)$. Since all $s_{i,j} = a$, we can take $i_0 = 0$ and we have $\sum _{0 < i\le n}s_{0,i}+1=na+1$.
\end{example}

Next we establish a lower bound for the initial degree of $\syz(J_{\cA})$. It turns out to be  a sharp lower bound for a large families of complete hypertetrahedral arrangements and for any triangular arrangement. We need to introduce some new notations.

Fix integers $i_0,q_0$ with $0 \leq i_0 < q_0 \leq n$. We denote by $T^{i_0,q_0}$ the set  of  all possibles $(n-1)$-uples $((i_0,j_0,q_0),(i_1,j_1,q_1), \hdots, (i_{n-2},j_{n-2},q_{n-2}))$ of triples $(i_m,j_m,q_m)$ of integers with $0 \leq i_{m},j_{m},q_{m} \leq n$ such that $j_{0} \notin \{i_0, q_{0}\}$ and, for each $1 \leq m \leq n-2$, the following two  conditions are satisfied:

\begin{itemize}
\item[(i)] $j_{m} \notin \{i_0,\hdots, i_{m-1}, j_{0}, \hdots, j_{m-1}, q_0, \hdots, q_{m-1}\}$,
\item[(ii)] $i_{m}  <  q_{m}$ and $i_{m},q_{m} \in \{i_0,\hdots, i_{m-1}, j_{0}, \hdots, j_{m-1}, q_0, \hdots, q_{m-1}\}$.
\end{itemize}

\noindent Each $(n-1)$-uple $v \in T^{i_0,q_0}$ uniquely determines a set
$$S^{i_0,q_0}_{v}:= \{s_{i_0,q_0}, s_{i_0,j_0}+ s_{j_0,q_0}, s_{i_1,j_1} + s_{j_1,q_1}, \hdots , s_{i_{n-2},j_{n-2}} + s_{j_{n-2},q_{n-2}}\},$$

\noindent and we define $m^{i_0,q_0}_{v} := \min \; S^{i_0,q_0}_{v}$ and $M^{i_0,q_0} := \max_{v \in T^{i_0,q_0}}\{m^{i_0,q_0}_{v}\}$.

Before presenting our result, let us illustrate the above notation with a couple of examples.

\begin{example}\rm (i) Let us fix $n = 2$ and $\cA$ a complete triangular arrangement in $\PP^2$. For
$i_0 = 1$ and $q_0 = 2$, $T^{1,2} = \{(1,0,2)\}$ and $M^{1,2} = \min\{s_{1,2},s_{1,0}+s_{0,2}\}$.

(ii) In $\PP^3$ consider the complete hypertetrahedral arrangement $\cA$ with defining equation $xyzt(x-y)(x-2z)(x^2-t^2)(y^3-z^3)(y^2-t^2)(z^3-t^3)$. Then $s_{0,1} = s_{0,2} = 1 $, $s_{0,3}  = s_{1,3} = 2$ and $s_{1,2}=s_{2,3} = 3$. It  follows directly that $M^{0,1} = M^{0,2} = 1$ and $M^{0,3} = M^{1,3} = 2$. We focus on determining $M^{1,2}$ and $M^{2,3}$. For $i_0 = 1$ and $q_0 = 2$, we have

$T^{1,2} = \{((1,0,2),(1,3,2)),((1,0,2),(0,3,1)),((1,0,2),(0,3,2)),((1,3,2),(1,0,3)), ((1,3,2), $ $(2,0,3))\},$

$$m_{v}^{1,2} = \begin{cases}
2, \quad \quad v \in T^{1,2} -\{(((1,3,2),(1,0,3)), ((1,3,2),(2,0,3))\}\\
3, \quad \quad \text{otherwise}.
\end{cases}$$
Thus $M^{1,2} = 3$. Analogously we compute $M^{2,3}$.

$T^{2,3} = \{((2,0,3),(2,1,3)),((2,0,3),(0,1,2)),((2,0,3),(0,1,3)),((2,1,3),(1,0,2)),((2,1,3), $ $(1,0,3))\}.$
$$m_{v}^{2,3} = \begin{cases}
2, \quad \quad v = ((2,1,3),(1,0,2))\\
3, \quad \quad \text{otherwise}.
\end{cases}$$
Hence $M^{2,3} = 3$.
\end{example}

\begin{theorem}
\label{ThmBound} Fix an integer $n\ge 2$ and $\cA \in \mathcal{H}(s_ {i,j})$. Let $D := \max_{0 \leq i_0 < q_0 \leq n}\{M^{i_0,q_0}\}$.  Then for all $1 \leq d \leq D$
$$\Der(-\log \cA )_d \subset R_{d-1}\theta_{E}.$$
\end{theorem}

\begin{proof} First of all, we observe that a derivation $\theta = \widetilde{f_0}\partial_{x_0}+ \cdots + \widetilde{f_n}\partial_{x_{n}} \in \Der(-\log \cA )_d$ if and only if for each $k$ with $0 \leq k \leq n$ there is $f_k \in R_{d-1}$ such that $\widetilde{f_{k}} = f_kx_k$ and, for each $i,j$ with $0 \leq i < j \leq n$ and each $r$ with $1 \leq r \leq s_{i,j}$ there is $f^{(r;i,j)} \in R_{d-1}$  such that
\begin{equation} \label{L-col1}  f^{(r;i,j)} L_{i,j}^{r}=\theta(L_{i,j}^{r})=a_i^{(r;i,j)} x_i f_i +a_j^{(r;i,j)} x_j f_j.  \end{equation}

Let us write
\[ f_t= \sum_{i_0+\cdots+i_n=d-1} \alpha^t_{(i_0, \hdots,i_n)}x_0^{i_0} x_1^{i_1} \cdots x_n^{i_n}, \quad \; 0 \leq t \leq n \]
and, for each $i,j$ with $0 \leq i < j \leq n$ and each $r$ with $1 \leq r \leq s_{i,j},$
\[ f^{(r;i,j)}= \sum_{i_0+\cdots+i_n=d-1} \alpha^{(r;i,j)}_{(i_0, \hdots,i_n)}x_0^{i_0} x_1^{i_1} \cdots x_n^{i_n}. \]
Our goal is to see that $\theta$ is a multiple of the Euler derivation $\theta_{E} \in \Der(-\log(\cA))_1$, that is to prove the equality $f_{0}=\dotsb=f_{n}$ or, equivalently, to see that for any partition $(i_0,\dotsc,i_{n})$ of $d-1$, and any $i,j$ with $0\leq i<j\leq n$, 
\begin{equation}\label{desired equalities}
  \alpha^{i}_{(i_0,\dotsc,i_{n})}=\alpha^{j}_{(i_0,\dotsc,i_{n}).}
\end{equation} To achieve this goal, we construct polynomials $P_{\ell}(X) \in k[X]$ of certain degree $\ell$ (see for instance (\ref{Eq:Polynomial})) having as coefficients the differences $\alpha^{i}_{(i_{0},\dotsc,i_{n})}-\alpha^{j}_{(i_{0},\dotsc,i_{n})}$. We show that such polynomials $P_{\ell}(X)$ have more than $\ell$ roots, thus obtaining the vanishing of all their coefficients, and so the desired equalities (\ref{desired equalities}) follow.

\vspace{4mm}
First, notice that equations (\ref{L-col1}) hold if and only if for any partition $(k_0, \hdots, k_n)$ of $d$ and for all $i,j,r$, with $0 \leq i < j \leq n$ and $1 \leq r \leq s_{i,j}$,
\[0=a_i^{(r;i,j)}(\alpha^{(r;i,j)}_{(k_0, \hdots,k_i-1, \hdots,k_n)}-\alpha^{i}_{(k_0, \hdots,k_i-1, \hdots,k_n)} )+ a_j^{(r;i,j)}(\alpha^{(r;i,j)}_{(k_0, \hdots,k_j-1, \hdots,k_n)}-\alpha^{j}_{(k_0, \hdots,k_j-1, \hdots,k_n)}), \]
with the convention that $\alpha^{(r;i,j)}_{(r_0, \hdots,r_i, \hdots,r_n)}=0$ and $\alpha^{t}_{(r_0, \hdots,r_i, \hdots,r_n)}=0$ if some $r_l<0$. For any partition $(k_0, \hdots, k_n)$ of $d$ and any pair $i<j$, we will write $$k_j=d-k_i-\sum_{p\neq i,j}k_p=d-k_i-\overline{k}_{j}^{i}.$$ According to this notation,
equations (\ref{L-col1}) hold if and only if for any partition $(k_0, \hdots, k_n)$  of $d$ and all $i,j,r$ with $0 \leq i < j \leq n$ and $1 \leq r \leq s_{i,j}$,

\begin{multline}\label{Eq:substitute1}
0=a_i^{(r;i,j)}(\alpha^{(r;i,j)}_{(k_0, \hdots,k_i-1, \hdots,k_n)}-\alpha^{i}_{(k_0, \hdots,k_i-1, \hdots,k_n)} )\\
+ a_j^{(r;i,j)}(\alpha^{(r;i,j)}_{(k_0, \hdots,d-k_i-\overline{k}_{j}^{i}-1, \hdots,k_n)}-\alpha^{j}_{(k_0, \hdots,d-k_i-\overline{k}_{j}^{i}-1, \hdots,k_n)} ).
\end{multline}
Since $a_j^{(r;i,j)} \neq 0$, for all partitions with $k_i=0$ we have
\begin{equation}\label{Eq:substitute ki=0}
    \alpha^{(r;i,j)}_{(k_0, \hdots,0, \hdots, d-\overline{k}_{j}^{i}-1, \hdots,k_n)}=\alpha^{j}_{(k_0, \hdots,0, \hdots, d-\overline{k}_{j}^{i}-1, \hdots,k_n)}. 
\end{equation}
 For any partition with $k_i=1$, and combining (\ref{Eq:substitute1}) and (\ref{Eq:substitute ki=0}), we get
 \begin{multline}\label{Eq:substitute ki=1}
 a_j^{(r;i,j)}\alpha^{(r;i,j)}_{(k_0, \hdots,1, \hdots, d-2-\overline{k}_{j}^{i}, \hdots,k_n)}= a_i^{(r;i,j)}(\alpha^{i}_{(k_0, \hdots,0, \hdots, d-1-\overline{k}_{j}^{i}, \hdots,k_n)}-\alpha^{j}_{(k_0, \hdots,0, \hdots, d-1-\overline{k}_{j}^{i}, \hdots,k_n)})\\
+ a_j^{(r;i,j)}\alpha^{j}_{(k_0, \hdots,1, \hdots, d-2-\overline{k}_{j}^{i}, \hdots,k_n)}.
 \end{multline}

Continuing with the same argument and using the fact that for any partition such that $k_i=d-\overline{k}_j^i$ we have $\alpha^{(r;i,j)}_{(k_0, \hdots,d-\overline{k}_j^i-1, \hdots,0, \hdots,k_n)}=\alpha^{i}_{(k_0, \hdots,d-\overline{k}_j^i-1, \hdots,0, \hdots,k_n)}$, we can conclude that

\[\alpha^{i}_{(k_0, \hdots,d-\overline{k}_j^i-1, \hdots,0, \hdots,k_n)}= \sum_{l=0}^{d-\overline{k}_j^i-2}(-1)^{l+1} \frac{(a_i^{(r;i,j)})^{d-\overline{k}_j^i-1-l}}{(a_j^{(r;i,j)})^{d-\overline{k}_j^i-l-1}}
( \alpha^i_{(k_0, \hdots, l, \hdots, d-\overline{k}_j^i -1-l, \hdots, k_n)}- \]
\[\alpha^j_{(k_0, \hdots, l, \hdots, d-\overline{k}_j^i -1-l, \hdots, k_n)} )  +  \alpha^j_{(k_0, \hdots, d-\overline{k}_j^i-1, \hdots, 0, \hdots, k_n)}. \]

Therefore,

\begin{equation}\label{Eq:PolynomialEvaluated}
\sum_{l=0}^{d-\overline{k}_j^i-1}(-1)^{l+1} (\frac{a_i^{(r;i,j)}}{a_j^{(r;i,j)} })^{d-\overline{k}_j^i-1-l}
( \alpha^i_{(k_0, \hdots, l, \hdots, d-\overline{k}_j^i -1-l, \hdots, k_n)}-\alpha^j_{(k_0, \hdots, l, \hdots, d-\overline{k}_j^i -1-l, \hdots, k_n)} )=0.
\end{equation}

Expression (\ref{Eq:PolynomialEvaluated}) can be seen as the evaluation at $\frac{a_i^{(r;i,j)}}{a_j^{(r;i,j)}}$ of the polynomial $P_{\ell}(X)\in k[X]$ of degree $\ell=d-\overline{k}_j^i-1$:
\begin{equation}\label{Eq:Polynomial}
P_{\ell}(X)=\sum_{l=0}^{d-\overline{k}_j^i-1}(-1)^{l+1}
( \alpha^i_{(k_0, \hdots, l, \hdots, d-\overline{k}_j^i -1-l, \hdots, k_n)}-\alpha^j_{(k_0, \hdots, l, \hdots, d-\overline{k}_j^i -1-l, \hdots, k_n)} )X^{d-\overline{k}_j^i-1-l}.
\end{equation}

Since the $s_{i,j}$ hyperplanes defined by $L_{i,j}^r = a_i^{(r;i,j)}x_{i}+a_j^{(r;i,j)}x_{j}$, $1 \leq r \leq s_{i,j}$, are pairwise different in $\PP^{n}$, we obtain that $P_{\ell}(X)$ has $s_{i,j}$ different roots. If $s_{i,j} \geq d-\overline{k}_j^i$, then $P_{\ell}(X)$ would have more roots than its degree $\ell$, implying that $P_{\ell}=0$. In particular, we obtain that for any $l$ with $0 \leq l \leq d-\overline{k}_j^i-1$,
\begin{equation} \label{L-AA} \alpha^i_{(k_0, \hdots, l, \hdots, d-\overline{k}_j^i -1-l, \hdots, k_n)}=\alpha^j_{(k_0, \hdots, l, \hdots, d-\overline{k}_j^i -1-l, \hdots, k_n)}.\end{equation}
On the contrary, if $s_{i,j} \leq d-\overline{k}_j^i-1$, we get the following expression
\[\alpha^j_{(k_0, \hdots, d-\overline{k}_j^i -1, \hdots,0, \hdots, k_n)}=  \sum_{l=0}^{d-\overline{k}_j^i-2}(-1)^{l+1} (\frac{a_i^{(r;i,j)}}{a_j^{(r;i,j)} })^{d-\overline{k}_j^i-1-l}
(\alpha^i_{(k_0, \hdots, l, \hdots, d-\overline{k}_j^i -1-l, \hdots, k_n)}-
\]
\begin{equation} \label{L-A} \alpha^j_{(k_0, \hdots, l, \hdots, d-\overline{k}_j^i -1-l, \hdots, k_n)} ) +\alpha^i_{(k_0, \hdots, d-\overline{k}_j^i -1, \hdots,0, \hdots, k_n)}.\end{equation}
\vspace{3mm}

\noindent {\bf Claim:} Let $k,j,i$ be pairwise different integers with $0 \leq k,j,i \leq n$. Assume that $s_{i,k} + s_{i,j} \geq d$ and that for any partition $(w):=(w_{0},\dotsc,w_{n})$ of $d-1$ we have $\alpha^{k}_{(w)}=\alpha^{j}_{(w)}$. Then, for any partition $(w)$ of $d-1$ we get
\[
\alpha^{i}_{(w)}=\alpha^{k}_{(w)}=\alpha^{j}_{(w)}.
\]
\vspace{3mm}

\noindent {\bf Proof of the Claim:} To simplify the notation, we will assume that $k=n-1$ and $j=n$. By hypothesis, $s_{i,n-1} + s_{i,n} \geq d$ and for any partition $(w)$ of $d-1$ we have $\alpha^{n-1}_{(w)}=\alpha^n_{(w)}$.  We also can assume that $s_{i,n-1} \leq s_{i,n}$, otherwise we can permute the role of $s_{i,n-1}$ by the one of $s_{i,n}$. Notice also that if $d \leq s_{i,n-1}$ or $d \leq s_{i,n}$, equation (\ref{L-AA}) would give us the desired equalities. So, we assume that $d > s_{i,n-1}$ and $d > s_{i,n}$,  we write $s_{i,n-1}=d-a$ and $s_{i,n}=d-b$ with $a \geq b > 0$ and,  we denote 
\[ \gamma^{(r;i,n-1)}:= \frac{a_i^{(r;i,n-1)}}{a_{n-1}^{(r;i,n-1)}} \quad \mbox{and} \quad \beta^{(r;i,n)}:= \frac{a_i^{(r;i,n)}}{a_{n}^{(r;i,n)}}.  \]
From (\ref{L-A}) we deduce that when $a \leq \overline{k}^{i}_{n-1}$, for any $0 \leq l \leq d-\overline{k}_{n-1}^i-1$
\begin{equation} \label{L-2}  \alpha^i_{(k_0, \hdots, d-\overline{k}_{n-1}^i -1-l, \hdots,l, k_n)}=\alpha^{n-1}_{(k_0, \hdots, d-\overline{k}_{n-1}^i -1-l, \hdots,l, k_n)}.
\end{equation}
If $a \geq \overline{k}^{i}_{n-1}+1$, we have
 \[  \sum_{l=0}^{d-\overline{k}_{n-1}^i-1}(-1)^{l} (\gamma^{(r;i,n-1)})^{l}
( \alpha^i_{(k_0, \hdots, d-\overline{k}_{n-1}^i -1-l, \hdots,l, k_n)}-\alpha^{n-1}_{(k_0, \hdots, d-\overline{k}_{n-1}^i -1-l, \hdots,l, k_n)} )=0. \]
Equivalently, when $b \leq \overline{k}^{i}_{n}$, for any $0 \leq l \leq, d-\overline{k}_{n}^i-1$
\begin{equation}  \alpha^i_{(k_0, \hdots, d-\overline{k}_{n}^i -1-l, \hdots, l)}=\alpha^{n}_{(k_0, \hdots, d-\overline{k}_n^i -1-l, \hdots, l)}
\end{equation}
and, if $b \geq \overline{k}^{i}_{n}+1$, we obtain
\[  \sum_{l=0}^{d-\overline{k}_{n}^i-1}(-1)^{l} (\beta^{(r;i,n)})^{l}
( \alpha^i_{(k_0, \hdots, d-\overline{k}_{n}^i -1-l, \hdots,l)}-\alpha^{n}_{(k_0, \hdots, d-\overline{k}_{n}^i -1-l, \hdots,l)} )=0. \]
The assumption $s_{i,n-1}+s_{i,n} \geq d$ implies that $a \leq d-b$. If $\overline{k}_{n}^{i} \leq b-1$, then we have $t \leq d-\overline{k}_n^i-1$ and the last equality can be written as
\[\sum_{l=0}^{a-1}(-1)^{l} (\beta^{(r;i,n)})^{l}
( \alpha^i_{(k_0, \hdots, d-\overline{k}_{n}^i -1-l, \hdots,l)}-\alpha^{n}_{(k_0, \hdots, d-\overline{k}_{n}^i -1-l, \hdots,l)} ) \]
 \begin{equation} \label{L-4} +   \sum_{l=t}^{d-\overline{k}_{n}^i-1}(-1)^{l} (\beta^{(r;i,j)})^{l}
( \alpha^i_{(k_0, \hdots, d-\overline{k}_{n}^i -1-l, \hdots,l)}-\alpha^{n}_{(k_0, \hdots, d-\overline{k}_{n}^i -1-l, \hdots,l)} )=0. \end{equation}
All the partitions appearing in the second addend of (\ref{L-4}) verify the conditions in (\ref{L-2}). Hence, for any $l$ with $t \leq l \leq  d-\overline{k}_n^i-1$, we have
\[\alpha^i_{(k_0, \hdots, d-\overline{k}_{n}^i -1-l, \hdots,l)}=\alpha^{n-1}_{(k_0, \hdots, d-\overline{k}_{n}^i -1-l, \hdots,l)}.   \]
By hypothesis, for any partition $(w)$ of $d-1$,  we have $\alpha_{(w)}^{n-1} = \alpha_{(w)}^{n}$. Thus, equation (\ref{L-4}) is reduced to
\[ 0= \sum_{l=0}^{t-1}(-1)^{l} (\beta^{(r;i,n)})^{l}
( \alpha^i_{(k_0, \hdots, d-\overline{k}_{n}^i -1-l, \hdots,l)}-\alpha^{n}_{(k_0, \hdots, d-\overline{k}_{n}^i -1-l, \hdots,l)} ). \]
This can be seen as the evaluation at $\beta^{(r;i,n)}$ of the polynomial $P_{t-1}(X)$of degree $t-1$: \[P_{t-1}(X) = 
\sum_{l=0}^{t-1}(-1)^{l} ( \alpha^i_{(k_0, \hdots, d-\overline{k}_{n}^i -1-l, \hdots,l)}-\alpha^{n}_{(k_0, \hdots, d-\overline{k}_{n}^i -1-l, \hdots,l)}) X^{l}.\]
Given that $t \leq s_{i,n}$, $P_{t-1}(X)$ has more than $t$ roots implying $P_{t-1} = 0$. Then, for any $l$ with $0 \leq l \leq t-1$, we get
\[\alpha^i_{(k_0, \hdots, d-\overline{k}_{n}^i -1-l, \hdots,l)}=\alpha^{n}_{(k_0, \hdots, d-\overline{k}_{n}^i -1-l, \hdots,l)} .  \]
Putting altogether, we obtain that for any partition $(w)$ of $d-1$ we have
\[ \alpha^i_{(w)}=\alpha^{n-1}_{(w)}=\alpha^{n}_{(w)}, \]
which finishes the proof of the claim.

\vspace{3mm}
Fix $v \in T^{n-1,n}$. Without loss of generality we can assume that  $d \leq \min m_{v}^{n-1,n}$, otherwise we permute the indexes in the argument below. Thus $d \leq s_{n-1,n}$ and
\[ d \leq \min \{s_{n-1,j_1}+s_{j_1,n}, s_{i_2,j_2}+s_{j_2,q_2},\hdots,  s_{i_{n-2},j_{n-2}}+s_{j_{n-2},q_{n-2}} \}. \]
Finally, using iteratively the {\bf Claim}, we will see that for any partition $(w) = (w_{0},\hdots, w_{n})$ of $d-1$, we have $\alpha^0_{(w)}= \alpha^1_{(w)}= \cdots = \alpha^n_{(w)}.$ 

Let us see the first steps of the iteration. Since $d \leq s_{n-1,n}$, from (\ref{L-AA}) it follows that $\alpha^{n-1}_{(w)}=  \alpha^n_{(w)}$, for any partition $(w)$. Moreover $d \leq s_{n-1,j_1}+s_{j_1,n}$ and, hence, by the claim we obtain equalities
\[  \alpha^{j_1}_{(w)}= \alpha^{n-1}_{(w)}=  \alpha^n_{(w)}. \]
We also know that $j_2 \notin \{j_1,n-1,n \} $ and $i_2,q_2 \in \{j_1,n-1,n\}$. Therefore, $\alpha^{i_2}_{(w)}= \alpha^{q_2}_{(w)}$ for any partition $(w)$ and, by assumption, $d \leq s_{i_2j_2}+s_{j_2q_2}$. Also by the {\bf Claim}
\[\alpha^{j_2}_{(w)}=  \alpha^{j_1}_{(w)}= \alpha^{n-1}_{(w)}=  \alpha^n_{(w)}, \]
for any partition $(w)$.
Now, $j_3 \notin \{i_2,j_1,j_2,q_2,n-1,n \} $ and $i_3,q_3 \in \{i_2,j_1,j_2,q_2,n-1,n\}$. Thus, $\alpha^{i_3}_{(w)}= \alpha^{q_3}_{(w)}$ for any partition $(w)$. Since $d \leq s_{i_3j_3}+s_{j_3q_3}$, applying  again the {\bf Claim} we obtain 
\[\alpha^{j_3}_{(w)}=\alpha^{j_2}_{(w)}=  \alpha^{j_1}_{(w)}= \alpha^{n-1}_{(w)}=  \alpha^n_{(w)}, \]
for any partition $(w)$.
Repeating the same argument for each summand $s_{i_m,j_m}+s_{j_m,q_m}$ in $S_{v}^{n-1,n}$, we obtain
\[\alpha^{j_m}_{(w)}= \cdots=\alpha^{j_2}_{(w)}=  \alpha^{j_1}_{(w)}= \alpha^{n-1}_{(w)}=  \alpha^n_{(w)}. \]
Given that we have $n-2$ summands and at each step $j_m \notin S^{n-1,n}_{m-1}$, after $n-2$ iterations we finally conclude that
\[\alpha^{0}_{(w)}=\alpha^{1}_{(w)}= \cdots=  \alpha^{n-1}_{(w)}=  \alpha^n_{(w)}, \]
for any partition $(w)$. In particular this implies that $\Der(-\log \cA )_d \subset R_{d-1} \theta_E$.
\end{proof}

\vspace{3mm}
\begin{remark}
\label{Remark-bound-Triangular}
 \rm  (1) When $n = 2$, the bound in Theorem \ref{ThmBound} simplifies to
$$M^{1,2} = \min \; S^{1,2} = \min\{s_{1,2}, s_{1,0} + s_{0,2}\}.$$
Indeed, $S^{0,1} = \min\{s_{0,1},s_{0,2}+s_{1,2}\} \leq s_{0,1}$ and $S^{0,2} = \min\{s_{0,2}, s_{0,1}+s_{1,2}\} \leq s_{0,2}$. Both are clearly smaller or equal than $s_{1,2}$ and $s_{1,0}+s_{0,2}$ anyway.

(2) Given integers $1 \leq s_{i,j}$, $ 0 \leq i < j \leq n$ such that $D = s_{n-1,n}$, there exists $\cA \in \cH(s_{i,j})$ with $\indeg(\syz(J_{\cA})) = D+1$.
Consider the extended Fermat arrangement $\cA_{s_{n-1,n}}$ with associated equation $x_{0}\cdots x_{n}\prod_{0 \leq i < j \leq n}(x_{i}^{s_{n-1,n}} - x_{j}^{s_{n-1,n}})$. For each $i,j$ with $0 \leq i < j \leq n$, we set $h_{i,j}$ to be the product of the $s_{i,j}$ different linear factors of $(x_{i}^{s_{n-1,n}}- x_{j}^{s_{n-1,n}})$. We claim that the arrangement $\cA$ with associated equation $x_{0}\cdots x_{n} \prod_{0 \leq i < j \leq n} h_{i,j}$ has $\indeg(\syz(J_{\cA})) = D+1$. Indeed, by Theorem \ref{ThmBound} we have that $\indeg(\syz(J_{\cA})) \geq D+1$. Since $\indeg(\syz(J_{\cA_{s_{n-1,n}}})) = D+1$ and $\syz(J_{\cA}) \subset \syz(\cA_{s_{n-1,n}})$, the claim follows.
\end{remark}

The following result collects examples of complete hypertetrahedral arrangements reaching the bound in Theorem \ref{ThmBound}.

\begin{proposition}\label{sharp}  For any integers $n\ge 2$,  $a\ge 2$ and $r\ge 1$ the hypertetrahedral arrangement $\cA _{a,n}^r$ with defining equation
$$x_0x_1\cdots x_n\prod _{0\le i<j \le n}(x_i^a-x_j ^a)g_{(r-1)a}(x_{n-1},x_n),$$ where $g_{(r-1)a}(x_{n-1},x_n)\in k[x_{n-1},x_n]$ is a general homogeneous polynomial of degree $(r-1)a$, is free with exponents  $(1,2a+1, \hdots , na+1, ra+1)$.
\end{proposition}
\begin{proof}  By Example \ref{extendedFermat}(ii),  we know that the result is true for $r=1$. Let us assume $r>1$ and write the equation $f_{\cA _{a,n}^r}$ of the arrangement $\cA _{a,n}^r$
as follows $$
\begin{array}{rcl} f_{\cA _{a,n}^r} & = & \displaystyle x_0x_1\cdots x_n\prod _{0\le i<j \le n}(x_i^a-x_j ^a)g_{(r-1)a}(x_{n-1},x_n) \\
 & = & f_{\cA _{a,n}^{r-1}}g_{a}(x_{n-1},x_n),
 \end{array}
 $$ where $g_{a}(x_{n-1},x_n)\in k[x_{n-1},x_n]$ is a homogeneous polynomial of degree $a$.
By Lemma  \ref{section}, any section $0\ne s_i\in \mbox{H}^0(\cT_{\cA_{a,n}^1}(-ia-1))$, $1\le i \le n-1$, induces a section $0\ne \sigma _i=s_i g_{a}(x_{n-1},x_n)\in \mbox{H}^0(\cT_{\cA_{a,n}^r}(-(i+1)a-1))$, being $\cT_{\cA_{a,n}^r}=\ker (\cO_{\PP^{n}}^{n+1} \longrightarrow J_{\cA _{a,n}^r}(\binom{n+1}{ 2}a+(r-1)a+n))$. Hence, we have two exact sequences to deal with:
\begin{equation}\label{jacobseq}
 0\longrightarrow \cT_{\cA_{a,n}^r}  \longrightarrow \cO_{\PP^{n}}^{n+1} \longrightarrow J_{\cA _{a,n}^r}(\binom{n+1}{2}a+(r-1)a+n)\longrightarrow 0
 \end{equation}
where $J_{\cA _{a,n}^r}$ is the Jacobian ideal of $f_{\cA _{a,n}^{r}}$, and
\begin{equation}\label{globsect}
0\longrightarrow \bigoplus  _{i=2}^n  \cO_{\PP^{n}} (-ia-1) \longrightarrow \cT_{\cA_{a,n}^r} \longrightarrow I_Z(-ra-1) \longrightarrow 0
\end{equation}
where $Z$ is a codimension 2 subscheme of $\PP^n$.

\noindent {\bf Claim:} $Z= \emptyset $.
In order to prove that $\deg (Z)=0$ we will compute the second Chern class of $\cT_{\cA_{a,n}^r}$ using two different approaches.
Using the exact sequence (\ref{globsect}), we get $c_t(\cT_{\cA_{a,n}^r})=c_t(\bigoplus  _{i=2}^n  \cO_{\PP^{n}} (-ia-1))c_t(I_Z(-ra-1)) $. Therefore, we have

\begin{equation}
\label{c2ONE}
\begin{array}{rcl}
c_2(\cT_{\cA_{a,n}^r}) & = & \displaystyle c_2(I_Z(-ra-1))+c_2(\bigoplus _{i=2}^n \cO_{\PP^{n}} (-ia-1))+\vspace{2mm}\\
& &  \displaystyle +c_1(I_Z(-ra-1))c_1(\bigoplus _{i=2}^n \cO_{\PP^{n}} (-ia-1)) \vspace{2mm}\\
& = & \displaystyle \deg (Z)+\sum _{2\le i<j\le n}(ia+1)(ja+1)+\sum _{i=2}^n(ia+1)(ra+1)\vspace{2mm}\\
& = & \deg (Z) +\binom{n}{2}+a(n-1)(\binom{n-1}{2}+r-1)+\frac{a^2}{2}(\binom{n+1}{3}\frac{3n+2}{2} +\\
& & +(r-1)(n-1)(n+2)).
\end{array}
\end{equation}

On the other hand, the second Chern class of $\cT_{\cA_{a,n}^r}$ can be computed using Jacobian map (\ref{jacobseq}) in addition to the codimension 2 part of the singular locus of $\cA _{a,n}^r$. The latter turns out to be:
\begin{itemize}
\item $\binom{n+1}{2}-1$ linear subspaces of codimension 2 which are the intersection of $a+2$ hyperplanes. Indeed, they correspond to the $a+2$ hyperplanes passing through the codimension 2 linear subspace $x_i=x_j=0$ with $0\le i <j\le n$ and $(i,j)\ne (n-1,n)$.
\item 1 linear subspace of codimension 2 which is the intersection of $ra+2$ hyperplanes. Indeed, the arrangement contains $ra+2$ hyperplanes containing the codimension 2 linear subspace $x_n=x_{n-1}=0$.
\item $\binom{n+1}{3}a^2$ linear subspaces of codimension 2 which are the intersection of 3 hyperplanes. Indeed, for $0\le i_1<i_2<i_3\le n$ and $\epsilon _1, \epsilon _2$ $a$-th roots of 1, the 3 hyperplanes $x_{1_1}-\epsilon _1x_{i_2}$, $x_{1_2}-\epsilon _2x_{i_3}$ and $x_{1_1}-\epsilon _1\epsilon _2x_{i_3}$ meet in a codimension 2 linear subspace.  
\item $\frac{a^2}{2}\binom{n+1}{2}\binom{n-1}{ 2}+(\binom{n+1}{2}+r-1)(n-1)a+(\binom{n+1}{ 2}-1)(r-1)a^2$ linear subspaces of codimension 2 which are the intersection of 2 hyperplanes.
\end{itemize}

Therefore, using the exact sequence (\ref{jacobseq}) and the properties of Chern classes (see \cite[Section 2.5]{Ful}),  we obtain:
\begin{equation}\label{c2BIS}
\begin{array}{rcl}
c_2(\cT_{\cA_{a,n}^r}) & = & - c_1( \cT_{\cA_{a,n}^r})c_1( J_{\cA _{a,n}^r}(\binom{n+1}{2}a+(r-1)a+n))-c_2( J_{\cA _{a,n}^r}(\binom{n+1}{2}a+(r-1)a+n))\\
& = & (\binom{n+1}{ 2}a+n+(r-1)a)^2-(\binom{n+1}{2}-1)(a+1)^2-(ra+1)^2\\
 & & -4\binom{n+1}{3}a^2 -\frac{a^2}{2}\binom{n+1}{2}\binom{n-1}{ 2}-(\binom{n+1}{2}+r-1)(n-1)a\\
 & & -(\binom{n+1}{2}-1)(r-1)a^2 \\
& = & \binom{n}{2}+a(n-1)(\binom{n-1}{
2}+r-1)+\frac{a^2}{2}(\binom{n+1}{3}\frac{3n+2}{2}\\
& & +(r-1)(n-1)(n+2)).
\end{array}
\end{equation}
Comparing the equalities (\ref{c2ONE}) and (\ref{c2BIS}),  we get $\deg(Z)=0$. Hence, the exact sequence  (\ref{globsect}) splits and  we conclude that $\cA_{a,n}^r$ is free with exponents $2a+1$, $\hdots $ , $na+1$, $ra+1$.
 \end{proof}

\vspace{0.3cm}
\noindent {\bf Triangular arrangements.}

\vspace{0.2cm}
In the last part of this section, we apply the previous results  to the particular case of complete triangular arrangements. From now on, we set
 $x=x_{0}$, $y=x_{1}$ and $z=x_{2}$ and $(s_1,s_2,s_3)=(s_{0,1},s_{0,2},
 s_{1,2})$ with $0 < s_{1} \leq s_{2} \leq s_{3}$. Any complete triangular arrangement $\cA\in \cH(s_1, s_2 ,s_3)$ of $s_1+s_2+s_3+3$  lines in $\PP^{2}$ is given by $x$, $y$, $z$ and the following $s_1+s_2+s_3$ linear forms:
 \[ l_i=x-a_iy, \quad 1 \leq i \leq s_1\]
  \[ l_j=b_jx-z, \quad 1 \leq j \leq s_2\]
  \[ l_t=y-c_tz, \quad 1 \leq t \leq s_3.\]

By Remark \ref{Remark-bound-Triangular}, $\indeg(\syz(J_{\cA}))\geq \min\{s_{3}, s_{1} + s_{2}\}+1$. Furthermore, we have the following result.

\begin{theorem}\label{triang} For any $\cA\in\cH(s_1,s_2,s_3)$ we have the following.
\begin{itemize}
\item[(i)] If $s_3\geq s_1+s_2$, then $\indeg(\syz(J_{\cA}))=s_1+s_2+1$. In other words, $\Der(-\log \cA )_d \subset R\theta_{E}$ for all $1 \leq d \leq \min(s_{1} + s_{2},s_3)$. 
\item[(ii)] If $s_3+1\leq s_1+s_2$, then $s_3+1\leq\indeg(\syz(J_{\cA}))\leq s_1+s_2+1$. Moreover, if $\indeg(\syz(J_{\cA}))=s_3+1$, then $\cA$ is free with exponents $(1,s_{3}+1,s_{1}+s_{2}+1)$.
\end{itemize}

\end{theorem}
\begin{proof}
See \cite[Theorem 1.2]{Dimca16}.  
\end{proof}

Let $\cA$ be a triangular arrangement in $\cH(s_{1},s_{2},s_{3})$ with $s_{3}+1\leq s_{1}+s_{2}$. By Theorem \ref{triang}, we know that there is an integer $s_{3}+1\leq d_{\cA}\leq s_{1}+s_{2}$ such that $\Der(-\log \cA)_{d_{\cA}}\nsubseteq R_{d_{\cA}-1}\theta_{E}$, but $\Der(-\log \cA)_{d}\subset R_{d-1}\theta_{E}$ for any $d\leq d_{\cA}$. Given an integer $s_{3}+1 \leq d \leq s_{1}+s_{2}$, we study which conditions must be satisfied by  a triangular arrangement $\cA\in\cH(s_{1},s_{2},s_{3})$ in order that $d_{\cA}=d$.

\begin{notation} \rm
Given  a complete triangular arrangement $\cA$ with $1\leq s_1\leq s_2\leq s_3$ and $s_3+1\leq s_1+s_2$, we denote 
\[f_{ijr}^{s}:=b_j^{-s}-(a_ic_r)^{s}.\]
For any integer $d$ in the range $s_3+1\leq d\leq s_1+s_2$, we define the matrices  $M_{\cA}^{d}$ as follows.

If $2d \leq s_1+s_2+s_3-1$, $M_{\cA}^{d}$ will be the matrix with $(d-s_1)(2d-s_2-s_3-1)$ columns and  $s_1s_2s_3$ rows $l_{ijr}$ given by
\[l_{ijr}=(g_{ijr}|h_{ijr}), \quad 1 \leq i \leq s_1, \quad  1 \leq j \leq s_2, \quad 1 \leq r \leq s_3,\quad\text{where}\]
\begin{flushleft}
{\scriptsize
$  \begin{array}{lllllllllll} g_{ijr}= (&f_{ijr}^{s_1} & \cdots & f_{ijr}^{d-1} &
c_rf_{ijr}^{s_1-1} & \cdots &c_rf_{ijr}^{d-2} & \cdots &c_r^{d-s_2-1}f_{ijr}^{s_1+s_2+1-d} & \cdots & c_r^{d-s_2-1} f_{ijr}^{s_2} \end{array})
$

$ \begin{array}{lllllllllll} h_{ijr}= (&f_{ijr}c_r^{s_1-1} & \cdots & f_{ijr}c_r^{d-2} & f_{ijr}^{2}c_r^{s_1-2} & \cdots & f_{ijr}^{2}c_r^{d-3} & \cdots &
 f_{ijr}^{d-s_3-1}c_r^{s_1+s_3+1-d} & \cdots & f_{ijr}^{d-s_3-1}c_r^{s_3}
 \end{array}).
 $
}
\end{flushleft}

 If $2d \geq  s_1+s_2+s_3$, we set
 {\footnotesize
 $$m=\frac{1}{2}(s_{1}(s_{1}-1)-(s_{2}+s_{3})(s_{2}+s_{3}+1-2d)).$$ }
In this case, $M_{\cA}^{d}$ will be the matrix with $m$  columns and  $s_1s_2s_3$ rows $l_{ijr}$ given by
\[l_{ijr}=(h_{ijr}|p_{ijr}|q_{ijr}|t_{ijr}), \quad 1 \leq i \leq s_1, \quad  1 \leq j \leq s_2, \quad 1 \leq r \leq s_3, \quad
\text{where}  \]

\begin{flushleft}
{\scriptsize
$ \begin{array}{lllllllllll} h_{ijr}= (&f_{ijr}c_r^{s_1-1} & \cdots & f_{ijr}c_r^{d-2} & f_{ijr}^{2}c_r^{s_1-2} & \cdots & f_{ijr}^{2}c_r^{d-3} & \cdots &
 f_{ijr}^{d-s_3-1}c_r^{s_1+s_3+1-d} & \cdots & f_{ijr}^{d-s_3-1}c_r^{s_3}
 \end{array}  )
$

$\begin{array}{lrrrllll} p_{ijr}= (&c_r^{s_1+s_3-d} f_{ijr}^{d-s_3} & \cdots & c_r^{s_1+s_3-d}f_{ijr}^{3d-s_1-s_2-2s_3-1} &
c_r^{s_1+s_3-d+1}f_{ijr}^{d-s_3} & \cdots \\
&&&\cdots&-c_r^{s_1+s_2-d+1}f_{ijr}^{3d-s_1-s_2-2s_3-2} & \cdots &c_r^{d-s_2-1}f_{ijr}^{d-s_3}) \end{array}
$

$ \begin{array}{lrrrrrrrrl} q_{ijr}=(&  f_{ijr}^{s_1} & \cdots & f_{ijr}^{2d-s_2-s_3-1} &
c_rf_{ijr}^{s_1-1} & \cdots &c_rf_{ijr}^{2d-s_2-s_3-2} & \cdots &c_r^{s_1+s_3-d-1}f_{ijr}^{d-s_3+1} &\cdots \\
&&&&&&&&\cdots & c_r^{s_1+s_3-d-1} f_{ijr}^{3d-s_1-s_2}) \end{array}
$

$ \begin{array}{lllllllllll} t_{ijr}= (&f_{ijr}^{2d-s_2-s_3} & \cdots & f_{ijr}^{d-1} &
c_rf_{ijr}^{2d-s_2-s_3-1} & \cdots &c_rf_{ijr}^{d-2} & \cdots &c_r^{d-s_2-1}f_{ijr}^{d-s_3+1} & \cdots & c_r^{d-s_2-1} f_{ijr}^{s_2} \end{array}).
$
 }
 \end{flushleft}

Finally, let us define the {\em expected rank} of $M_{\cA}^{d}$ to be

{\scriptsize
\[
\exprk(M_{\cA}^{d})=
\left\{
\begin{array}{l}
(d-s_{1})(2d-s_{2}-s_{3}-1),\quad\text{if}\quad 2d\leq s_{1}+s_{2}+s_{3}-1\\
\frac{1}{2}(s_{1}(s_{1}-1)-(s_{2}+s_{3})(s_{3}+s_{3}+1-2d),\quad\text{if}\quad 2d\geq s_{1}+s_{2}+s_{3}.
\end{array}
\right.
\]
}
\end{notation}

Keeping this notation, we can determine when the arrangement $\cA$ has a degree $d$ derivation in the range $s_3+1\leq d\leq s_1+s_2$.

\begin{theorem}
\label{trunyo}
Let $\cA$ be a complete triangular arrangement with $0\leq s_1\leq s_2\leq s_3$ and assume that $s_3+1\leq s_1+s_2$.  Let $s_3+1\leq d\leq s_1+s_2$. Then,
$\Der(-\log \cA)_{d}\subseteq R_{d-1}\theta_{E}$ if and only if $M_{\cA}^{d}$ has maximal rank.
\end{theorem}
\begin{proof}

Let us recall the proof of Theorem \ref{ThmBound} in the case $n=2$. A derivation $\theta\in \Der(-\log \cA )_{d}$ is identified with a triple $(f_{1},f_{2},f_{3})\in R_{d-1}^{3}$ such that, writing $f_{i}=\sum_{w_{1}+w_{2}+w_{3}=d-1}\alpha^{i}_{(w_{1},w_{2},w_{3})}x^{w_{1}}y^{w_{2}}z^{w_{3}}$ the following conditions are satisfied:
 $$  \sum_{l=0}^{d-1-k} a_{i}^{l} (\alpha ^1_{(l,d-1-k-l,k)} -\alpha ^2_{(l,d-1-k-l,k)}) = 0 \text{ for } 0\le k \le d-1 \text{ and } 1\le i \le s_1,$$
 $$ \sum_{l=0}^{d-1-k} b_{j}^{l} (\alpha ^1_{(d-1-k-l,k,l)} -\alpha ^3_{(d-1-k-l,k,l)}) = 0 \text{ for } 0\le k \le d-1 \text{ and } 1\le j\le s_2,$$
 $$ \sum_{l=0}^{d-1-k} c_{r}^{l} (\alpha ^2_{(k,l,d-1-k-l)} -\alpha ^3_{(k,l,d-1-k-l)}) = 0\text{ for } 0\le k \le d-1, \text{ and } 1\le r\le s_3,$$
 $$ \alpha ^2_{(k,l,d-1-k-l)}=\alpha ^3_{(k,l,d-1-k-l)} \text{ for all } 0\le k \le d-1 \text{ and  } 0\le l\le d-1-k,$$
 $$\alpha ^1_{(l,d-1-k-l,k)}=\alpha ^2_{(l,d-1-k-l,k)} \text{ for all } d-s_1\le k \le d-1 \text{ and  } 0\le l\le d-1-k, \text{ and } $$
 $$\alpha ^1_{(d-1-k-l,k,l)}=\alpha ^3_{(d-1-k-l,k,l)} \text{ for all } d-s_2\le k \le d-1 \text{ and  } 0\le l\le d-1-k.$$

 Let us now define, for any partition $(w_{1},w_{2},w_{3})$ of $d-1$, 
 $$X^{1,2}_{(w_{1},w_{2},w_{3})} =\alpha ^1_{(w_{1},w_{2},w_{3})} -\alpha ^2_{(w_{1},w_{2},w_{3})},$$
 $$X^{1,3}_{(w_{1},w_{2},w_{3})} =\alpha ^1_{(w_{1},w_{2},w_{3})} - \alpha ^3_{(w_{1},w_{2},w_{3})} \text{ and }$$
 $$X^{2,3}_{(w_{1},w_{2},w_{3})} =\alpha ^2_{(w_{1},w_{2},w_{3})} - \alpha ^2_{(w_{1},w_{2},w_{3})}.$$

 Hence, any derivation $\theta\in \Der(-\log \cA )_{d}$ corresponds bijectively with a solution of the following homogeneous linear system:

\begin{equation}\label{System}
\left\{
\begin{array}{l}
 \sum_{l=0}^{d-1-k} a_i^lX_{(l,d-1-k-l,k)}^{12}=0, \quad \mbox{for} \quad 1 \leq i \leq s_1\\
 \sum_{l=0}^{d-1-k} b_j^lX_{(d-1-k-l,k,l)}^{13}=0, \quad \mbox{for} \quad 1 \leq j \leq s_2 \\
 \sum_{l=0}^{d-1-k} c_r^lX_{(k,l,d-1-k-l)}^{23}=0,\quad \mbox{for} \quad 1 \leq r \leq s_3 \\
 X_{(l,d-1-k-l,k)}^{12}=0 \quad \mbox{for} \quad d-s_1 \leq k \leq d-1 \quad \mbox{and} \quad 0 \leq l \leq d-1-k \\
 X_{(d-1-k-l,k,l)}^{13}=0 \quad \mbox{for} \quad d-s_2 \leq k \leq d-1 \quad \mbox{and} \quad 0 \leq l \leq d-1-k \\
 X_{(k,l,d-1-k-l)}^{23}=0 \quad \mbox{for} \quad d-s_3 \leq k \leq d-1 \quad \mbox{and} \quad 0 \leq l \leq d-1-k \\
 X_{(w_{1},w_{2},w_{3})}^{13}=X_{(w_{1},w_{2},w_{3})}^{12}+X_{(w_{1},w_{2},w_{3})}^{23} \quad \mbox{for} \quad w_{1}+w_{2}+w_{3}=d-1.
 \end{array}
 \right.
\end{equation}

Moreover, notice that $\Der(-\log \cA )_{d}\nsubseteq R_{d-1}\theta_{E}$ if and only if (\ref{System}) has some solution different from zero. 
Our strategy will be to determine when the matrix associated to the system (\ref{System}) has maximal rank.
First, from (\ref{System}) we deduce the following equalities:

\begin{equation} \label{x12}  X^{12}_{(0,d-1-k,k)}=-\sum_{l=1}^{d-1-k} a_i^lX_{(l,d-1-k-l,k)}^{12}, \quad \mbox{for} \quad 1 \leq i \leq s_1
\quad \mbox{and} \quad 0 \leq k \leq d-s_1-1 \end{equation}
\begin{equation}  \label{x13} X_{(0,k,d-1-k)}^{13}=- \sum_{l=1}^{d-k-1} b_j^{-l}X_{(l,k,d-1-k-l)}^{13}, \quad \mbox{for} \quad 1 \leq j \leq s_2
\quad \mbox{and} \quad 0 \leq k \leq d-s_2-1 \end{equation}
\begin{equation} \label{x23} X_{(k,0,d-1-k)}= - \sum_{l=1}^{d-1-k} c_r^lX_{(k,l,d-1-k-l)}^{23},\quad \mbox{for} \quad 1 \leq r \leq s_3
\quad \mbox{and} \quad 0 \leq k \leq d-s_3-1. \end{equation}
In addition, for any $l$ with $0 \leq l \leq d-1-k$,
\begin{equation} \label{0x12}  X_{(l,d-1-k-l,k)}^{12}=0, \quad X_{(l,d-1-k-l,k)}^{23}=X_{(l,d-1-k-l,k)}^{13} \quad \mbox{for} \quad d-s_1 \leq k \leq d-1   \end{equation}
\begin{equation} \label{0x13} X_{(d-1-k-l,k,l)}^{13}=0, \quad X_{(d-1-k-l,k,l)}^{23}=-X_{(d-1-k-l,k,l)}^{12} \quad \mbox{for} \quad d-s_2 \leq k \leq d-1  \end{equation}
\begin{equation} \label{0x23} X_{(k,l,d-1-k-l)}^{23}=0, \quad X_{(k,l,d-1-k-l)}^{12}=X_{(k,l,d-1-k-l)}^{13} \quad \mbox{for} \quad d-s_3 \leq k \leq d-1 \end{equation}
and
\begin{equation} \label{lhsrhs}
\sum_{l=0}^{d-2}b_j^{l-d+1}X_{(d-1-l,0,l)}^{13}=\sum_{l=1}^{d-1}c_r^{l}X_{(0,l,d-1-l)}^{23},\quad \mbox{for}\quad 1\leq j\leq s_2\quad\mbox{and}\quad 1\leq r\leq s_3.
\end{equation}

Therefore, a solution of equations from (\ref{x12}) to (\ref{0x23}) corresponds to a derivation in $\Der(-\log \cA )_{d}$ if and only if it satisfies the condition (\ref{lhsrhs}). In order to study this last equation we need to distinguish two different cases:

\noindent {\bf Case 1: $2d \leq  s_1+s_2+s_3-1$.}

First of all, we will analyze the left hand side of (\ref{lhsrhs}). To this end, we write it as
\[ \begin{array}{ll} L:= \sum_{l=0}^{d-2}b_j^{l-d+1}X_{(d-1-l,0,l)}^{13} & =\sum_{s=1}^{d-1}b_j^{-s}X_{(s,0,d-1-s)}^{13} \\
&= \sum_{s=1}^{d-s_3-1}b_j^{-s}X_{(s,0,d-1-s)}^{13}+ \sum_{s=d-s_3}^{s_1-1}b_j^{-s}X_{(s,0,d-1-s)}^{13} \\ & +\sum_{s=s_1}^{d-1}b_j^{-s}X_{(s,0,d-1-s)}^{13} \\
& = \sum_{s=1}^{d-s_3-1}b_j^{-s}X_{(s,0,d-1-s)}^{23} +\sum_{s=s_1}^{d-1}b_j^{-s}X_{(s,0,d-1-s)}^{12},
\end{array}\]
where the last equality follows from (\ref{0x13}) and (\ref{0x23}). Now using (\ref{x23}) and again equalities (\ref{0x12}) and (\ref{0x13}) we get
 \[ \begin{array}{ll} L
& = \sum_{s=1}^{d-s_3-1}b_j^{-s}[ - \sum_{l=1}^{d-s_2-1}c_r^lX^{23}_{(s,l,d-1-s-l)} -\sum_{l=s_1-s}^{d-s-1}c_r^lX^{23}_{(s,l,d-1-s-l)} ] \\ & +\sum_{s=s_1}^{d-1}b_j^{-s}X_{(s,0,d-1-s)}^{12} \\
& = -\sum_{s=1}^{d-s_3-1}  \sum_{l=1}^{d-s_2-1}b_j^{-s}c_r^lX^{13}_{(s,l,d-1-s-l)}  +\sum_{s=1}^{d-s_3-1} \sum_{l=s_1-s}^{d-s-1}b_j^{-s}c_r^lX^{12}_{(s,l,d-1-s-l)} \\ & +\sum_{s=s_1}^{d-1}b_j^{-s}X_{(s,0,d-1-s)}^{12}. \\
\end{array}\]
Now we will analyze the right hand side of (\ref{lhsrhs}). To this end, we write it as
\[ \begin{array}{ll} R:= \sum_{l=1}^{d-1}c_r^{l}X_{(0,l,d-1-l)}^{23} &  =\sum_{s=1}^{d-1}c_r^{s}X_{(0,s,d-1-s)}^{23} \\
&= \sum_{s=1}^{d-s_2-1}c_r^sX^{13}_{(0,s,d-1-s)}- \sum_{s=s_1}^{d-1}c_r^sX^{12}_{(0,s,d-1-s)} \\
& =\sum_{s=1}^{d-s_2-1}c_r^sX^{13}_{(0,s,d-1-s)}- \sum_{s=0}^{d-s_1-1}c_r^{d-1-s}X^{12}_{(0,d-1-s,s)}, \end{array} \]
where we have used equalities from (\ref{0x12}) to (\ref{0x23}) and we have performed a change of variables in the last equality.
Using (\ref{x13}) and (\ref{x12}) and having in mind the ranges  where the variables vanish (given by (\ref{0x12}) to (\ref{0x23})),  we get
\[ \begin{array}{ll} R &=  - \sum_{s=1}^{d-s_2-1} \sum_{l=1}^{d-s_3-1}c_r^sb_j^{-l}X^{13}_{(l,s,d-1-s-l)} -
\sum_{s=1}^{d-s_2-1} \sum_{l=s_1-s}^{d-s-1}c_r^sb_j^{-l}X^{12}_{(l,s,d-1-s-l)} \\
&+  \sum_{s=0}^{d-s_1-1} \sum_{l=1}^{d-s_3-1}c_r^{d-1-s}a_i^{l}X^{12}_{(l,d-1-s-l,s)} +
\sum_{s=0}^{d-s_1-1} \sum_{l=s_2-s}^{d-s-1}c_r^{d-1-s}a_i^{l}X^{12}_{(l,d-1-s-l,s)}. \end{array} \]
Writing together both expressions, we obtain:
{
  \[ \begin{array}{lll}
 -\sum_{s=1}^{d-s_3-1}  \sum_{l=1}^{d-s_2-1}b_j^{-s}c_r^lX^{13}_{(s,l,d-1-s-l)} & & \\ -\sum_{s=1}^{d-s_3-1} \sum_{l=s_1-s}^{d-s-1}b_j^{-s}c_r^lX^{12}_{(s,l,d-1-s-l)}& & \\ +\sum_{s=s_1}^{d-1}b_j^{-s}X_{(s,0,d-1-s)}^{12} & & \\
 & = &  - \sum_{s=1}^{d-s_2-1} \sum_{l=1}^{d-s_3-1}c_r^sb_j^{-l}X^{13}_{(l,s,d-1-s-l)} \\ & &  -
\sum_{s=1}^{d-s_2-1} \sum_{l=s_1-s}^{d-s-1}c_r^sb_j^{-l}X^{12}_{(l,s,d-1-s-l)} \\
& & +  \sum_{s=0}^{d-s_1-1} \sum_{l=1}^{d-s_3-1}c_r^{d-1-s}a_i^{l}X^{12}_{(l,d-1-s-l,s)} \\ & &  +
\sum_{s=0}^{d-s_1-1} \sum_{l=s_2-s}^{d-s-1}c_r^{d-1-s}a_i^{l}X^{12}_{(l,d-1-s-l,s)}. \end{array} \]
}
Reordering the equation and performing suitable changes of coordinates we get:
{
  \[ \begin{array}{l}
 \sum_{s=1}^{d-s_3-1} \sum_{l=s_1-s}^{d-s-1}c_r^{l}[b_j^{-s}-(a_ic_r)^{s} ]X^{12}_{(s,l,d-1-s-l)} +  \\
 \sum_{l=0}^{d-s_2-1} \sum_{s=s_1-l}^{d-l-1}c_r^{l}[b_j^{-s}-(a_ic_r)^{s} ]X^{12}_{(s,l,d-1-s-l)}  =0. \end{array} \]
}

Hence, $\Der(-\log \cA)\subset R_{d-1}\theta_{E}$ if and only if $M_{\cA}^d$ has maximal rank.

\noindent {\bf Case 2: $2d \geq  s_1+s_2+s_3$}.

Reasoning analogously as in the above case, we deduce that the condition (\ref{lhsrhs}) is equivalent to

  \[ \begin{array}{l}
 \sum_{s=1}^{d-s_3-1} \sum_{l=s_1-s}^{d-s-1}c_r^{l}[b_j^{-s}-(a_ic_r)^{s} ]X^{12}_{(s,l,d-1-s-l)} +  \\
 \sum_{l=s_1+s_3-d}^{d-s_2-1} \sum_{s=d-s_3}^{2d-s_2-s_3-l-1}c_r^{l}[b_j^{-s}-(a_ic_r)^{s} ]X^{12}_{(s,l,d-1-s-l)} +  \\
  \sum_{l=0}^{s_1+s_3-d-1} \sum_{s=s_1-l}^{2d-s_2-s_3-l-1}c_r^{l}[b_j^{-s}-(a_ic_r)^{s} ]X^{12}_{(s,l,d-1-s-l)} +  \\
 \sum_{l=0}^{d-s_2-1} \sum_{s=2d-s_2-s_3-l}^{d-l-1}c_r^{l}[b_j^{-s}-(a_ic_r)^{s} ]X^{12}_{(s,l,d-1-s-l)} =0. \end{array} \]

Thus, $\Der(-\log \cA)\subset R_{d-1}\theta_{E}$ if and only if $M_{\cA}^d$ has maximal rank.
\end{proof}

We end this paper with a characterization of free triangular arrangements.

\begin{corollary}
Let $\cA\in\cH(s_{1},s_{2},s_{3})$ be a triangular arrangement with $s_{3}+1\leq s_{1}+s_{2}$. Then, $\cA$ is free with exponents $(1,s_{3}+1+k,s_{1}+s_{2}+1-k)$ with $k\leq \frac{s_{1}+s_{2}-s_{3}}{2}$ if and only if
{
\[
rk(M_{\cA}^{d})=
\left\{
\begin{array}{ll}
\exprk(M_{\cA}^{d}),&\text{if}\quad s_{3}+1\leq d\leq s_{3}+k\\
\exprk(M_{\cA}^{d})-\binom{2+d-(s_{3}+1+k)}{2},&\text{if}\quad s_{3}+1+k\leq d\leq s_{1}+s_{2}-k\\
\exprk(M_{\cA}^{d})-\binom{2+d-(s_{3}+1+k)}{2}-\binom{2+d-(s_{1}+s_{2}+1-k)}{2},&\text{if}\quad s_{1}+s_{2}+1-k\leq d\leq s_{1}+s_{2}+1.
\end{array}
\right.
\]
}
\end{corollary}

\begin{proof}
It follows directly from Theorem \ref{trunyo}.
\end{proof}


\end{document}